\documentclass[11pt,a4paper,reqno]{article}

\usepackage[a4paper,includeheadfoot, total={6in,10in}]{geometry}
\usepackage[affil-it]{authblk}
\usepackage[dvipsnames]{xcolor}
\usepackage[numbers,sort]{natbib}
\usepackage[sc]{mathpazo}
\usepackage[T1]{fontenc}
\usepackage{eulervm}
\usepackage[utf8]{inputenc}
\usepackage{amsfonts,amsmath,amssymb,amsthm,
bbm,chngcntr,enumerate,graphicx,mathtools,tikz, relsize}

\usepackage[linktocpage=true]{hyperref}

\usepackage{subcaption} 
\usetikzlibrary{positioning}
\usetikzlibrary{decorations.text, decorations.markings, mindmap,trees,shadows, decorations.pathreplacing}
\tikzstyle{doublearr}=[latex-latex,red, line width=0.5pt]
\tikzstyle{doublearr2}=[latex-latex,green!80!black, line width=0.5pt]

\tikzstyle{mybox} = [draw=black, thick, minimum height=.4cm,
    rectangle]
\tikzstyle{fancytitle} =[fill=blue, text=white]
\usetikzlibrary{mindmap,trees,shadows}

\usepackage[titletoc,page,header]{appendix}

\def\sss{ }

\newcommand{\cmt}[1]{[\texttt{{\color{red}#1}}]}

\newcommand{\JSQ}{\mathrm{JSQ}}
\newcommand{\Er}{\mathrm{Er}}
\newcommand{\hEr}{\hat{\mathrm{Er}}}

\newcommand{\BS}{\mathbf{S}}
\newcommand{\ff}{\mathbf{f}}
\newcommand{\gb}{\mathbf{g}}
\newcommand{\FF}{\mathbf{F}}

\newcommand{\yy}{\mathbf{y}}

\newcommand{\qq}{\mathbf{q}}

\newcommand{\QQ}{\mathbf{Q}}

\newcommand{\ZZ}{\mathbf{Z}}
\newcommand{\bQ}{\bar{Q}}
\newcommand{\hQ}{\hat{Q}}

\newcommand{\Op}{\mathrm{O}_{\sss P}}
\newcommand{\op}{\mathrm{o}_{\sss P}}

\newcommand{\mmu}{\boldsymbol{\mu}}

\long\def\/*#1*/{}
\newcommand{\e}{\mathrm{e}}
\newcommand\given[1][]{\:#1\middle|\:} 
\newcommand\Pro[1]{\mathbbm{P}\left(#1\right)}  
\newcommand\E[1]{\mathbbm{E}\left(#1\right)}  
\newcommand\norm[1]{\left\|#1\right\|}  
\newcommand{\pto}{\ensuremath{\xrightarrow{\mathbbm{P}}}}  
\newcommand{\dto}{\ensuremath{\xrightarrow{\mathcal{L}}}}  
\newcommand{\dif}{\ensuremath{\mbox{d}}}

\newcommand\ind[1]{\ensuremath{\mathbbm{1}_{\left[#1\right]}}} 
\newcommand{\R}{\mathbbm{R}}                 
\newcommand{\N}{\mathbbm{N}}                 
\newcommand{\Z}{\mathbbm{Z}}				
\newcommand{\bZ}{\bar{\mathbbm{Z}}}				
\makeatletter 
\newcommand{\fixed@sra}{$\vrule height 2\fontdimen22\textfont2 width 0pt\shortrightarrow$}
\newcommand{\sa}[1]{%
  \mathrel{\text{\rotatebox[origin=c]{\numexpr#1*45}{\fixed@sra}}}
}

\newtheorem{theorem}{Theorem}
\newtheorem{corollary}[theorem]{Corollary}
\newtheorem{lemma}[theorem]{Lemma}
\newtheorem{proposition}[theorem]{Proposition}
\newtheorem{remark}[theorem]{Remark}
\newtheorem{definition}{Definition}
\newtheorem*{theorem*}{Theorem}
\newtheorem{assumption}{Assumption}

\numberwithin{equation}{section}
\numberwithin{theorem}{section}

\hypersetup{
 colorlinks,
 linkcolor=red,          
 citecolor=blue,       
 filecolor=red,   
 urlcolor=black, 
 pdftitle={},
 pdfauthor={},
 pdfcreator={Debankur Mukherjee},
 pdfsubject={},
 pdfkeywords={}
}

\begin{document}

\title{Asymptotic Optimality of Power-of-$d$ Load Balancing \\
in Large-Scale Systems}

\author[1]{Debankur Mukherjee\footnote{\texttt{d.mukherjee@tue.nl}}}
\author[1,2]{Sem C.~Borst}
\author[1]{\\ Johan S.H.~van Leeuwaarden}
\author[3]{Philip A.~Whiting}
\affil[1]{
Eindhoven University of Technology,
 The Netherlands}
\affil[2]{Nokia Bell Labs, Murray Hill, NJ, USA}
\affil[3]{
Macquarie University, North Ryde, NSW, Australia}

\renewcommand\Authands{ , }

\date{\today}

\maketitle 

\begin{abstract}
{

We consider a system of $N$ identical server pools and a single dispatcher where tasks arrive as a Poisson process of rate $\lambda(N)$. Arriving tasks cannot be queued, and must immediately be assigned to one of the server pools to start execution, or discarded. The execution times are assumed to be exponentially distributed with unit mean, and do not depend on the number of other tasks receiving service. However, the experienced performance (e.g.~in terms of received throughput) does degrade with an increasing number of concurrent tasks at the same server pool. The dispatcher therefore aims to evenly distribute the tasks across the various server pools. Specifically, when a task arrives, the dispatcher assigns it to the server pool with the minimum number of tasks among $d(N)$ randomly selected server pools. This assignment strategy is called the JSQ$(d(N))$ scheme, as it resembles the power-of-$d$ version of the Join-the-Shortest-Queue (JSQ) policy, and will also be referred to as such in the special case $d(N) = N$.
 
We construct a stochastic coupling to bound the difference in the system occupancy processes between the JSQ policy and a scheme with an arbitrary value of $d(N)$. We use the coupling to derive the fluid limit in case $d(N) \to \infty$ and $\lambda(N)/N \to \lambda$ as $N \to \infty$, along with the associated fixed point. The fluid limit turns out to be insensitive to the exact growth rate of $d(N)$, and coincides with that for the JSQ policy. We further leverage the coupling to establish that the diffusion limit corresponds to that for the JSQ policy as well, as long as $d(N)/\sqrt{N} \log(N) \to \infty$, and characterize the common limiting diffusion process. These results indicate that the JSQ optimality can be preserved at the fluid-level and diffusion-level while reducing the overhead by nearly a factor O($N$) and O($\sqrt{N}/\log(N)$), respectively.
}
\end{abstract}
\newpage
\tableofcontents

\section{Introduction}

In the present paper we establish asymptotic optimality for a broad
class of randomized load balancing strategies.
While the specific features of load balancing policies may considerably
differ, the principal purpose is to distribute service requests or tasks
among servers or distributed resources in parallel-processing systems.
Well-designed load balancing schemes provide an effective mechanism
for improving relevant performance metrics experienced by users
while achieving high resource utilization levels.
The analysis and design of load balancing policies has attracted
strong renewed interest in the last several years, mainly motivated
by significant challenges involved in assigning tasks
(e.g.~file transfers, compute jobs, data base look-ups)
to servers in large-scale data centers, see for instance~\cite{Ananta13}.

Load balancing schemes can be broadly categorized as static (open-loop),
dynamic (closed-loop), or some intermediate blend, depending on the
amount of real-time feedback or state information (e.g.~queue
lengths or load measurements) that is used in assigning tasks.
Within the category of dynamic policies, one can further distinguish
between push-based and pull-based approaches, depending on whether the
initiative resides with a dispatcher actively collecting feedback from the
servers, or with the servers advertizing their availability or load status.
The use of state information naturally allows dynamic policies
to achieve better performance and greater resource pooling gains,
but also involves higher implementation complexity and potentially
substantial communication overhead.
The latter issue is particularly pertinent in large-scale data centers,
which deploy thousands of servers and handle massive demands,
with service requests coming in at huge rates.

In the present paper we focus on a basic scenario of $N$~identical
parallel server pools and a single dispatcher where tasks arrive
as a Poisson process.
Incoming tasks cannot be queued, and must immediately be dispatched
to one of the server pools to start execution, or discarded.
The execution times are assumed to be exponentially distributed,
and do not depend on the number of other tasks receiving service,
but the experienced performance (e.g.~in terms of received throughput
or packet-level delay) does degrade in a convex manner with
an increasing number of concurrent tasks.
These characteristics pertain for instance to video streaming sessions
and various interactive applications.
In contrast to elastic data transfers or computing-intensive jobs,
the duration of such sessions is hardly affected by the number
of contending service requests.
The perceived performance in terms of video quality or packet-level
latency however strongly varies with the number of concurrent tasks,
creating an incentive to distribute the incoming tasks across the
various server pools as evenly as possible.

Specifically, adopting the usual time scale separation assumption,
suppose that the task-perceived performance at a particular server pool
can be described as some function $F(x)$ of the instantaneous number
of concurrent tasks~$x$, and let $X = (X_1, \dots, X_N)$,
with $X_n$ the number of active tasks at the $n^{\mathrm{th}}$ server pool.
Then $G(X) = \sum_{n = 1}^{N} X_n F(X_n) / \sum_{n = 1}^{N} X_n$,
provides a proxy for the instantaneous overall system performance.
In many situations, the function $G(\cdot)$ tends to be either
Schur-convex of Schur-concave.
For example, if $F(\cdot)$ is convex increasing (for instance average
packet-level delay), then $G(\cdot)$ is Schur-convex, i.e.,
$G(X) \leq G(Y)$ if $X$ is majorized by $Y$, i.e., $X$ is `more balanced' 
than $Y$ with $\sum_{n = 1}^{N} X_n = \sum_{n = 1}^{N} Y_n$.
Likewise, if $F(X_n) = U(1 / X_n)$, with $U(\cdot)$ is concave
increasing (for instance throughput utility), then $G(\cdot)$ is
Schur-concave, i.e., $G(X) \geq G(Y)$ if $X$ is majorized by $Y$.

The so-called Join-the-Shortest-Queue (JSQ) policy has primarily been
considered for load balancing among parallel \emph{single-server
queues} where it furnishes several strong optimality guarantees
\cite{EVW80,W78,Winston77}.
Its fundamental ability of optimally balancing tasks across parallel
resources also translates however into crucial optimality properties with
respect to the performance criterion $G(\cdot)$ in the present context with infinite-server dynamics.
In particular, let $X^\Pi(t) = (X_1^\Pi(t), \dots, X_N^\Pi(t))$,
with $X_n^\Pi(t)$ denoting the number of active tasks at the $n^{\mathrm{th}}$ server
pool at time~$t$ under a task assignment scheme~$\Pi$.
Then, given the same initial conditions and in the absence of any blocking,
$\big\{X^{\JSQ}(t)\big\}_{t \geq 0}$ is majorized by $\big\{X^\Pi(t)\big\}_{t \geq 0}$
for any non-anticipating
task assignment scheme~$\Pi$ \cite{towsley,Towsley95,STC93, M87, MS91}.
Thus $G(X^{\JSQ}(t))$ is either stochastically smaller or larger than
$G(X^\Pi(t))$ at all times~$t$ for any task assignment scheme~$\Pi$,
depending on whether the function $G(\cdot)$ is Schur-convex
or Schur-concave.
In a scenario where each server pool can only accommodate a maximum
of $B < \infty$ simultaneous tasks, the JSQ policy belongs to the class
of policies that stochastically minimize the total cumulative number
of blocked tasks over any time interval~$[0, t]$~\cite{Towsley1992, STC93}.

In order to implement the JSQ policy, a dispatcher requires
instantaneous knowledge of the numbers of tasks at all the server pools,
which may give rise to a substantial communication burden,
and may not be scalable in scenarios with large numbers of server pools.
The latter issue has motivated consideration of so-called JSQ($d$)
policies, where the dispatcher assigns an incoming task to a server pool
with the minimum number of active tasks among $d$~randomly selected
server pools.
Mean-field limit theorems in Mitzenmacher~\cite{Mitzenmacher01}
and Vvedenskaya {\em et al.}~\cite{VDK96} indicate that even a value
as small as $d = 2$ yields significant performance improvements
in a single-server queueing regime with $N \to \infty$,
in the sense that the tail of the queue length distribution at each
individual server falls off much more rapidly compared to a strictly
random assignment policy ($d = 1$).
This is commonly referred to as the ``power-of-two'' effect.
Work of Turner~\cite{T98} and recent papers by Mukhopadhyay {\em et al.}~\cite{MKMG15,MMG15}
and Xie {\em et al.}~\cite{XDLS15} have shown similar power-of-two
properties for loss probabilities in a \emph{blocking} scenario with infinite-server dynamics as
described above.

As illustrated by the above, the diversity parameter~$d$ induces
a fundamental trade-off between the amount of communication overhead
and the performance in terms of blocking probabilities or throughputs.
For example, a strictly random assignment policy can be implemented
with zero overhead, but for any finite buffer capacity $B < \infty$
the blocking probability does \emph{not} fall to zero as $N \to \infty$.
In contrast, a nominal implementation of the JSQ policy (without
maintaining state information at the dispatcher) involves O($N$)
overhead per task, but it can be shown that for any subcritical load,
the blocking probability vanishes as $N \to \infty$.
As mentioned above, JSQ($d$) strategies with a fixed parameter $d \geq 2$
yield significant performance improvements over purely random task
assignment while reducing the overhead by a factor O($N$) compared
to the JSQ policy.
However, the blocking probability does \emph{not} vanish in the limit,
and in that sense a fixed value of~$d$ is not sufficient to achieve
asymptotically optimal performance.

In order to gain further insight in the trade-off between performance
and communication overhead as governed by the diversity parameter~$d$,
we also consider a regime where the number of servers~$N$ grows large,
but allow the value of~$d$ to depend on~$N$,
and write $d(N)$ to explicitly reflect that.
For convenience, we assume a Poisson arrival process of rate $\lambda(N)$
and unit-mean exponential service requirements.

We construct a stochastic coupling to bound the difference
in the system occupancy processes between the JSQ policy and a scheme
with an arbitrary value of $d(N)$.
We exploit the coupling to obtain the fluid limit
in case $\lambda(N) / N \to \lambda < B$ and $d(N)\to\infty$
as $N \to \infty$, along with the associated fixed point.
As it turns out, the fluid limit is insensitive to the exact growth
rate of $d(N)$, and in particular coincides with that for the ordinary JSQ
policy.
This implies that the overhead of the JSQ policy can be reduced by
almost a factor O($N$) while maintaining fluid-level optimality.

We further consider the diffusion limit of the system occupancy states, and consider the infinite-server dynamics analog of the Halfin-Whitt regime.
We leverage the above-mentioned coupling to prove that the diffusion limit
in case $d(N) /(\sqrt{N} \log(N))\to\infty$ corresponds to that for the ordinary JSQ policy,
and characterize the common limiting diffusion process.
This indicates that the overhead of the JSQ policy can be reduced by
almost a factor O($\sqrt{N}/\log(N)$) while retaining diffusion-level optimality.

The above results mirror fluid-level and diffusion-level optimality properties reported in the companion paper~\cite{MBLW16-3}
for the power-of-d($N$) strategies in a scenario with single-server queues.
As it turns out, however, the infinite-server dynamics in the present paper require a fundamentally different coupling argument to establish asymptotic equivalence.
In particular, for the single-server dynamics, first the servers are ordered according to the number of active tasks, and the departures at the ordered servers under two different policies are then coupled.
In contrast, for the infinite-server dynamics, the departure rate at the ordered server pools can vary depending on the exact number of active tasks.
Therefore, the departure processes under two different policies cannot be coupled as before, which necessitates the construction of a novel stochastic coupling.
Specifically, one can think of the coupling for the single-server dynamics as one-dimensional (depending only upon the ordering of the servers), while the coupling we introduce in this paper is two-dimensional, with the server ordering as one coordinate and the number of tasks as the other, as will be explained in greater detail later.
We further elaborate on the necessity and novelty of the coupling methodology developed in the current paper, and reflect on the contrast with the stochastic optimality results for the JSQ policy in the
existing literature and the coupling technique in~\cite{MBLW16-3} in Remarks~\ref{rem:novelty} and \ref{rem:contrast}.
In addition, 
the fluid- and diffusion-limit results in the infinite-server scenario are also notably different from those in~\cite{MBLW16-3}.
More specifically, we extend the fluid-limit result in~\cite[Theorem~4.1]{MBLW16-3} to a more general class of assignment probabilities and departure rate functions, and depending on whether the scaled arrival rate converges to an integer or not,  obtain a qualitatively different behavior of the occupancy state process on diffusion scale.
Furthermore, the diffusion limit result in \cite[Theorem~2.4]{MBLW16-3} characterizes the diffusion-scale behavior only in the transient regime, whereas in the current paper we are able to analyze the steady-state behavior as well.

The remainder of the paper is organized as follows.
In Section~\ref{sec: model descr} we present a detailed model description, and provide an overview of the main results. 
In Section~\ref{sec:eqiv} we explain the proof outline and introduce a notion of asymptotic equivalence of two assignment schemes. 
Section~\ref{sec:proof-equiv} introduces a stochastic coupling between any two schemes, and proves the asymptotic equivalence results.
Sections~\ref{sec:proof-equiv}--\ref{sec:integral} contain the proofs of the main results, and in Section~\ref{sec:performance} we reflect upon various performance implications. We conclude in Section~\ref{sec:conclusion} with some pointers to open problems and future research.

\section{Main Results}\label{sec: model descr}
\label{sec:main}

\subsection{Model description and notation}
Consider a system with $N$~parallel identical server pools and a single
dispatcher where tasks arrive as a Poisson process of rate~$\lambda(N)$.
Arriving tasks cannot be queued, and must immediately be assigned
to one of the server pools to start execution.
The execution times are assumed to be exponentially distributed with unit
mean, and do not depend on the number of other tasks receiving service.
Each server pool is however only able to accommodate a maximum
of $B$~simultaneous tasks (possibly $B = \infty$),
and when a task is allocated to a server pool that is already handling
$B$~active tasks, it gets permanently discarded.

Specifically, when a task arrives, the dispatcher assigns it to the
server pool with the minimum number of active tasks among $d(N)$
randomly selected server pools ($1 \leq d(N) \leq N$).
As mentioned earlier, this assignment strategy is called a JSQ$(d(N))$
scheme, as it closely resembles the power-of-$d$ version
of the  Join-the-Shortest-Queue (JSQ) policy, and will also
consisely be referred to as such in the special case $d(N) = N$.
We will consider an asymptotic regime where the number of server pools~$N$
and the task arrival rate $\lambda(N)$ grow large in proportion,
with $\lambda(N) / N \to \lambda\leq B$ as $N \to \infty$.
For convenience, we denote $K = \lfloor \lambda \rfloor$
and $f = \lambda - K \in [0, 1)$.

For any $d(N)$ ($1 \leq d(N) \leq N$), let $\QQ^{\sss d(N)}(t) =
(Q_1^{\sss d(N)}(t), Q_2^{\sss d(N)}(t), \dots, Q_B^{\sss d(N)}(t))$
be the system occupancy state, where $Q_i^{\sss d(N)}(t)$ is the number
of server pools under the JSQ($d(N)$) scheme with $i$~or more active
tasks at time~$t$, $i = 1, \dots, B$. 
A schematic diagram of the $Q_i$-values is provided in Figure~\ref{fig:1}.
We occasionally omit the superscript $d(N)$, and replace it by~$N$, to refer
to the $N^{\mathrm{th}}$ system, when the value of $d(N)$ is clear from the context.
In case of a finite buffer size $B < \infty$, when a task is discarded,
we call it an \emph{overflow} event, and we denote by $L^{\sss d(N)}(t)$ the
total number of overflow events under the JSQ($d(N)$) policy up to time~$t$. 

Throughout we assume that at each arrival epoch the server pools are ordered
in  nondecreasing order of the number of active tasks (ties can be broken arbitrarily), see Figure~\ref{fig:1}, and whenever we refer to some ordered server pool, it should be understood with respect to this prior ordering, unless mentioned otherwise.

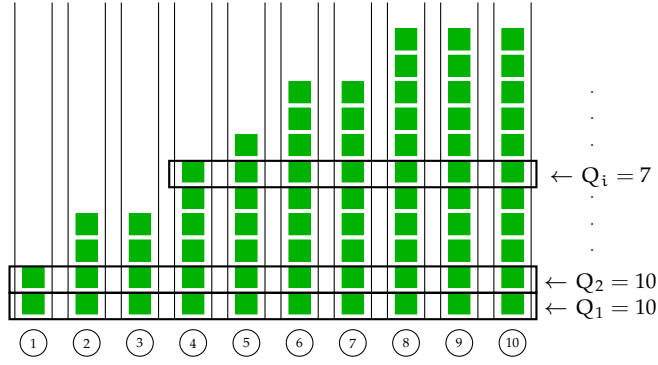
\begin{figure}
\begin{center}
\begin{tikzpicture}[scale=.70]
\foreach \x in {10, 9,...,1}
	\draw (\x,6)--(\x,0)--(\x+.7,0)--(\x+.7,6);
\foreach \x in {10, 9,...,1}
	\draw (\x+.35,-.45) node[circle,inner sep=0pt, minimum size=10pt,draw] {{{\tiny $\mathsmaller{\x}$}}} ;
\foreach \y in {0, .5}
	\draw[fill=green!70!black,green!70!black] (1.15,.1+\y) rectangle (1.55,.5+\y);
\foreach \y in {0, .5, 1, 1.5}
	\draw[fill=green!70!black,green!70!black] (2.15,.1+\y) rectangle (2.55,.5+\y);
\foreach \y in {0, .5, 1, 1.5}
	\draw[fill=green!70!black,green!70!black] (3.15,.1+\y) rectangle (3.55,.5+\y);
\foreach \y in {0, .5, 1, 1.5, 2, 2.5}
	\draw[fill=green!70!black, green!70!black] (4.15,.1+\y) rectangle (4.55,.5+\y);
\foreach \y in {0, .5, 1, 1.5, 2, 2.5, 3}
	\draw[fill=green!70!black,green!70!black] (5.15,.1+\y) rectangle (5.55,.5+\y);
\foreach \y in {0, .5, 1, 1.5, 2, 2.5, 3, 3.5, 4}
	\draw[fill=green!70!black,green!70!black] (6.15,.1+\y) rectangle (6.55,.5+\y);
\foreach \y in {0, .5, 1, 1.5, 2, 2.5, 3, 3.5, 4}
	\draw[fill=green!70!black,green!70!black] (7.15,.1+\y) rectangle (7.55,.5+\y);
\foreach \y in {0, .5, 1, 1.5, 2, 2.5, 3, 3.5, 4, 4.5, 5}
	\draw[fill=green!70!black, green!70!black] (8.15,.1+\y) rectangle (8.55,.5+\y);
\foreach \y in {0, .5, 1, 1.5, 2, 2.5, 3, 3.5, 4, 4.5, 5}
	\draw[fill=green!70!black,green!70!black] (9.15,.1+\y) rectangle (9.55,.5+\y);
\foreach \y in {0, .5, 1, 1.5, 2, 2.5, 3, 3.5, 4, 4.5, 5}
	\draw[fill=green!70!black,green!70!black] (10.15,.1+\y) rectangle (10.55,.5+\y);

\draw[thick] (.9,0) rectangle (10.8,.5);
\draw[thick] (.9,.5) rectangle (10.8,1);
\draw[thick] (3.9,2.5) rectangle (10.8,3);

\draw  (12, .2) node {{\scriptsize $\leftarrow Q_1=10$}};
\draw  (12, .7) node {{\scriptsize $\leftarrow Q_2=10$}};

\draw  (11.85, 1.3) node {{\tiny $\cdot$}};
\draw  (11.85, 1.8) node {{\tiny $\cdot$}};
\draw  (11.85, 2.3) node {{\tiny $\cdot$}};
\draw  (12, 2.7) node {{\scriptsize $\leftarrow Q_i=7$}};
\draw  (11.85, 3.3) node {{\tiny $\cdot$}};
\draw  (11.85, 3.8) node {{\tiny $\cdot$}};
\draw  (11.85, 4.3) node {{\tiny $\cdot$}};

\end{tikzpicture}
\end{center}
\caption{The occupancy state of the system; When the server pools are arranged in nondecreasing order of the number of active tasks, $Q_i$  represents the width of the $i^\mathrm{th}$ row, as shown above.}
\label{fig:1}
\end{figure}

Boldfaced letters will be used to denote vectors.
A sequence of random variables $\big\{X_N\big\}_{N\geq 1}$ is said to be $\Op(g(N))$, or $\op(g(N))$, for some function $g:\N\to\R_+$, if the sequence of scaled random variables $\big\{X_N/g(N)\big\}_{N\geq 1}$ is  a tight sequence, or converges to zero in probability, respectively. 
Whenever we mention `with high probability', it should be understood as `with probability tending to 1 as the underlying scaling parameter tends to infinity'.
For stochastic boundedness of a process we refer to \cite[Definition 5.4]{PTRW07}. 
Also, $f$ will be called `diverging to infinity' if $g(N)\to\infty$ as $N\to\infty$.
For any complete separable metric space $E$, denote by $D_E[0,\infty)$, the set of all $E$-valued c\'adl\'ag (right continuous with left limit exists) processes.
By the symbol `$\dto$' we denote convergence in distribution for real-valued random variables, and  with respect to Skorohod-$J_1$ topology for \emph{c\'adl\'ag} processes.

\subsection{Fluid-limit results}\label{ssec:fluid}

In order to state the fluid-limit results, we first introduce some
useful notation.
Denote the fluid-scaled system occupancy state by
$\qq^{\sss d(N)}(t) := \QQ^{\sss d(N)}(t) / N$.
We will denote by $\tilde{S}=\big\{\QQ\in\Z^B:Q_i \leq Q_{i-1} \mbox{ for all } i = 2, \dots, B\big\}$ and $S =\big\{\qq \in [0, 1]^B: q_i \leq q_{i-1} \mbox{ for all } i = 2, \dots, B\big\}$
 the set of all possible unscaled and fluid-scaled occupancy states, respectively.
Further define $S^N:= S\cap\big\{i/N:1\leq i\leq N\big\}^B$ as the space of all fluid-scaled occupancy states of the $N^\mathrm{th}$ system.
We take the following product norm on $S$: for $\mathbf{q}_1=(q_{1,1},q_{1,2},\ldots, q_{1,B})$, $\mathbf{q}_2=(q_{2,1},q_{2,2},\ldots, q_{2,B})\in S$,
$$\rho(\mathbf{q}_1,\mathbf{q}_2):=\sum_{i=1}^{B}\frac{|q_{1,i}-q_{2,i}|}{2^i},$$
and  all the convergence results below will be with respect to product topology.
We often write $\rho(\qq_1,\qq_2)$ as $\norm{\qq_1-\qq_2}$. 
Let $(E,\hat{\rho})$ be a metric space.
We call a function $g:S\to E$  Lipschitz continuous on $S$, if there exists $L>0$, such that for all $x,y\in S,$
$$\hat{\rho}(g(x),g(y))\leq L \|x-y\|.$$
For any $\qq \in S$, denote by $m(\qq) = \min\big\{i: q_{i + 1} < 1\big\}$
the minimum number of active tasks among all server pools, with the convention that
$q_{B+1} = 0$ if $B < \infty$.
Now distinguish two cases, depending on whether the normalized arrival
rate $\lambda$ is larger than $m(\qq)(1 - q_{m(\qq) + 1})$ or not.
If $\lambda \leq m(\qq) (1 - q_{m(\qq) + 1})$, then define 
$$p_{m(\qq) - 1}(\qq) = 1,\quad\mbox{and}\quad p_i(\qq) = 0\quad\mbox{for all}\quad i \neq m(\qq) - 1.$$
On the other hand, if $\lambda >m(\qq) (1 - q_{\sss m(\qq) + 1})$,
then 
\begin{equation}
p_{i}(\qq)=
\begin{cases}
m(\qq)(1 - q_{\sss m(\qq) + 1})/\lambda & \quad\mbox{ for }\quad i=m(\qq)-1,\\
1 - p_{\sss m(\qq) - 1}(\qq) & \quad\mbox{ for }\quad i=m(\qq),\\
0&\quad \mbox{ otherwise.}
\end{cases}
\end{equation}
Note that the assumption $\lambda \leq B$ ensures that the latter case
cannot occur when $B<\infty$ and $m(\qq) = B$.

\begin{theorem}{\normalfont (Universality of fluid limit for JSQ$(d(N))$ scheme)}
\label{fluidjsqd}
Assume $\qq^{\sss d(N)}(0) \pto \qq^\infty \in S$ as $N \to \infty$.
For the JSQ($d(N)$) scheme with $d(N)$ diverging to infinity, the sequence of processes
$\big\{\qq^{\sss d(N)}(t)\big\}_{t \geq 0}$ has a weak limit $\big\{\qq(t)\big\}_{t \geq 0}$ that
satisfies the system of integral equations
\[
q_i(t) = q_i(0)+
\lambda\int_0^t p_{i-1}(\qq(s))\dif s - i\int_0^t (q_i(s) - q_{i+1}(s))\dif s, \quad i=1,\ldots, B,
\]
where $\qq(0) = \qq^\infty$ and the coefficients $p_i(\cdot)$ are
as defined above.
\end{theorem}

The above theorem shows that the fluid-level dynamics do not depend
on the specific growth rate of $d(N)$ as long as $d(N) \to\infty$
as $N \to \infty$.
In particular, the JSQ$(d(N))$ scheme with $d(N) \to\infty$ as $N \to\infty$ exhibits
the same behavior as the ordinary JSQ policy, and thus achieves
fluid-level optimality. This result can be intuitively interpreted as follows. Since $d(N)$ is growing, for large $N$, at an arrival epoch, if the fraction of server pools with the minimum number of active tasks becomes positive, then with high probability at least one of the $d(N)$ selected server pools will be from the ones with the minimum number of active tasks.
This ensures that as long as $d(N)\to\infty$ as $N \to\infty$, the difference in $Q_i$-values between the ordinary JSQ policy and the JSQ$(d(N))$ scheme can not become $O(N)$, yielding fluid-level optimality.

The coefficient $p_i(\qq)$ represents the fraction
of incoming tasks assigned to server pools with exactly~$i$ active
tasks in the fluid-level state $\qq \in S$.
Assuming $m(\qq) < B$, a strictly positive fraction $1 - q_{m(\qq) + 1}$
of the server pools have exactly $m(\qq)$ active tasks.
Since $d(N)\to\infty$ as $N \to\infty$, the fraction of incoming tasks that get assigned
to server pools with $m(\qq) + 1$ or more active tasks is therefore zero:
$p_i(\qq) = 0$ for all $i = m(\qq) + 1, \dots, B - 1$.
Also, tasks at server pools with exactly~$i$ active tasks are completed
at (normalized) rate $i(q_i - q_{i + 1})$, which is zero for all
$i = 1, \dots, m(\qq) - 1$, and hence the fraction of incoming tasks
that get assigned to server pools with $m(\qq) - 2$ or less active
tasks is zero as well: $p_i(\qq) = 0$ for all $i = 0, \dots, m(\qq) - 2$.
This only leaves the fractions $p_{m(\qq) - 1}(\qq)$
and $p_{m(\qq)}(\qq)$ to be determined.
Now observe that the fraction of server pools with exactly $m(\qq) - 1$
active tasks is zero.
However, since tasks at server pools with exactly $m(\qq)$ active tasks
are completed at (normalized) rate $m(\qq) (1 - q_{m(\qq) + 1}) > 0$,
incoming tasks can be assigned to server pools with exactly $m(\qq) - 1$
active tasks at that rate.
We thus need to distinguish between two cases, depending on whether
the normalized arrival rate $\lambda$ is larger than
$m(\qq) (1 - q_{m(\qq) + 1})$ or not.
If $\lambda \leq m(\qq) (1 - q_{m(\qq) + 1})$, then all the incoming tasks
can be assigned to server pools with exactly $m(\qq) - 1$ active tasks,
so that $p_{m(\qq) - 1}(\qq) = 1$ and $p_{m(\qq)}(\qq) = 0$.
On the other hand, if $\lambda > m(\qq )(1 - q_{m(\qq) + 1})$, then not
all incoming tasks can be assigned to server pools with exactly
$m(\qq) - 1$ active tasks, and a positive fraction will be assigned
to server pools with exactly $m(\qq)$ active tasks:
$p_{m(\qq) - 1}(\qq) = m(\qq) (1 - q_{m(\qq) + 1}) / \lambda$
and $p_{m(\qq)}(\qq) = 1 - p_{m(\qq) - 1}(\qq)$.

It is easily verified that the unique fixed point of the differential
equation in Theorem~\ref{fluidjsqd} is given by
\begin{equation}
\label{eq:fixed point}
q_i^\star = \left\{\begin{array}{ll} 1 & i = 1, \dots, K \\
f & i = K + 1 \\
0 & i = K + 2, \dots, B, \end{array} \right.
\end{equation}
and thus $\sum_{i=1}^B q_i^\star =\lambda$.
This is consistent with the results in Mukhopadhyay
{\em et al.}~\cite{MKMG15,MMG15} and Xie {\em et al.}~\cite{XDLS15}
for fixed~$d$, where taking $d \to \infty$ yields the same fixed point.
However, the results in \cite{MKMG15,MMG15,XDLS15} for fixed~$d$
cannot be directly used to handle joint scalings, and do not yield the
universality of the entire fluid-scaled sample path
for arbitrary initial states as established in Theorem~\ref{fluidjsqd}.

Having obtained the fixed point of the fluid limit, we now establish the interchange of the mean-field 
$(N\to\infty)$ and stationary $(t\to\infty)$ limits. Let 
$$\pi^{\sss d(N)}(\cdot)=\lim_{t\to\infty}\Pro{\qq^{\sss d(N)}(t)=\cdot}$$ 
be the stationary measure of the occupancy states of the $N^{\mathrm{th}}$ system. 

\begin{proposition}[{Interchange of limits}]
\label{prop:interchange}
The sequence of stationary measures $\big\{\pi^{\sss d(N)}\big\}_{N\geq 1}$ with $d(N)\to\infty$ as $N \to\infty$ converges weakly to $\pi^\star$, where $\pi^\star=\delta_{\qq^\star}$ with $\delta_x$ being the Dirac measure concentrated upon $x$, and $\qq^\star$ defined by~\eqref{eq:fixed point}.
\end{proposition}
\begin{proof}
Observe that $\pi^{\sss d(N)}$ is defined on $S$, and $S$ is a compact set when endowed with the product topology. Prohorov's theorem implies that the sequence $\big\{\pi^{\sss d(N)}\big\}_{N\geq 1}$ is relatively compact, and hence, has a convergent subsequence. Let $\big\{\pi^{\sss d(N_n)}\big\}_{n\geq 1}$ be a convergent subsequence, with $\big\{N_n\big\}_{n\geq 1}\subset\N$, such that $\pi^{\sss d(N_n)}\dto\hat{\pi}$. We show that $\hat{\pi}$ is unique and equals the measure $\pi^\star.$

First of all note that if $\qq^{\sss d(N_n)}(0)\sim\pi^{\sss d(N_n)}$, then $\qq^{\sss d(N_n)}(t)\sim\pi^{\sss d(N_n)}$. Also, the fact that $\qq^{\sss d(N_n)}(t)\dto\qq(t)$, and $\pi^{\sss d(N_n)}\dto\hat{\pi}$, means that $\hat{\pi}$ is a fixed point of the deterministic process $\big\{\qq(t)\big\}_{t\geq 0}$. Since the latter fixed point is unique, $\qq^\star$, we can conclude the desired convergence of the stationary measure.
\end{proof}

\subsection{Diffusion-limit results for non-integral \texorpdfstring{$\boldsymbol{\lambda}$}{lambda}}

As it turns out, the diffusion-limit results may be qualitatively
different, depending on whether $f = 0$ or $f > 0$,
and we will distinguish between these two cases accordingly.
Observe that for any assignment scheme, in the absence of overflow events, the total number of active
tasks evolves as the number of jobs in an M/M/$\infty$ system
with arrival rate $\lambda(N)$ and unit service rate,
for which the diffusion limit is well-known~\cite{Robert03}.
For the JSQ$(d(N))$ scheme with $d(N)/(\sqrt{N} \log(N))\to\infty$ as $N\to\infty$, we can
establish, for suitable initial conditions, that the total number of server
pools with $K - 2$ or less and $K + 2$ or more tasks is negligible
on the diffusion scale.
If $f > 0$, the number of server pools with $K - 1$ tasks is negligible
as well, and the dynamics of the number of server pools with $K$
or $K + 1$ tasks can then be derived from the known diffusion limit
of the total number of tasks mentioned above.
In contrast, if $f = 0$, the number of server pools with $K - 1$ tasks
is not negligible on the diffusion scale, and the limiting behavior is
qualitatively different, but can still be characterized.

We first consider the case $f > 0$, and define $f(N):= \lambda(N)-KN.$
Based on the above observations,
we define the following centered and scaled processes:
\begin{equation}
\begin{split}
\bar{Q}_i^{\sss d(N)}(t) &:= \frac{N-Q^{\sss d(N)}_i(t)}{\sqrt{N}}\geq 0,\quad i \leq K,\\  
\bar{Q}_{K+1}^{\sss d(N)}(t) &:= \frac{Q^{\sss d(N)}_{K+1}(t)-f(N)}{\sqrt{N}}\in\R,\\
\bar{Q}_i^{\sss d(N)}(t) &:=Q^{\sss d(N)}_i(t)\geq 0,\quad\text{for}\quad i \geq K+2.
\end{split}
\end{equation}
\begin{theorem}
{\normalfont (Universality of diffusion limit for JSQ$(d(N))$ scheme, $f > 0$)}
\label{th:diff pwr of d 1}
If $f > 0$, $\bQ^{\sss d(N)}_{K+1}(0)\pto\bQ_{K+1}\in\R$, $\bQ^{\sss d(N)}_i(0)\pto 0$ for $i\neq K+1$, and $d(N)/(\sqrt{N} \log(N))\to\infty$ as $N\to\infty$, then the following holds
as $N \to \infty$:
\begin{enumerate}[{\normalfont(i)}]
\item For $i \leq K$, $\big\{\bar{Q}_i^{\sss d(N)}(t)\big\}_{t \geq 0} \dto
\big\{\bar{Q}_i(t)\big\}_{t \geq 0}$, where $\bar{Q}_i(t) \equiv 0$.
\item $\big\{\bar{Q}_{K+1}^{\sss d(N)}(t)\big\}_{t \geq 0} \dto 
\big\{\bar{Q}_{K+1}(t)\big\}_{t \geq 0}$, where $\bar{Q}_{K+1}(t)$ is given
by the Ornstein-Uhlenbeck process satisfying the following stochastic
differential equation:
\begin{equation}
\label{eq:OU process1}
\dif \bar{Q}_{K+1}(t) =
- \bar{Q}_{K+1}(t) \dif t + \sqrt{2 \lambda} \dif W(t),
\end{equation}
where $W(t)$ is the standard Brownian motion.
\item For $i \geq K+2$, $\big\{\bar{Q}_i^{\sss d(N)}(t)\big\}_{t \geq 0} \dto 
\big\{\bar{Q}_i(t)\big\}_{t \geq 0}$, where $\bar{Q}_i(t) \equiv 0$.
\end{enumerate}

\end{theorem}
Loosely speaking, the above theorem says that, if $f > 0$
and $d(N)/(\sqrt{N} \log(N))\to\infty$ as $N\to\infty$, then over any finite time horizon,
there will only be $o_P(\sqrt{N})$ server pools with fewer than
$K$ or more than $K+1$~active tasks, and $fN+O_P(\sqrt{N})$ server pools with precisely
$K + 1$ active tasks.
Also, as long as $d(N)/(\sqrt{N} \log(N))\to\infty$ as $N\to\infty$, the JSQ$(d(N))$
scheme exhibits the same behavior as the ordinary JSQ policy (i.e., $d(N)=N$),
and thus achieves diffusion-level optimality. 
The result can be heuristically explained as follows. 
When the number of server pools with the minimum number of active tasks is $O(\sqrt{N})$, the JSQ$(d(N))$ scheme should be able to assign the incoming tasks
with high probability to one of the server pools with the minimum number of active tasks. 
To be able to select one of the $O(\sqrt{N})$ server pool out of $N$ server pools, $d(N)$ must grow faster than $\sqrt{N}$. 
Now further observe that in any finite time interval there are on average $O(N)$ arrivals, and hence it is not enough to assign the incoming task to the appropriate server pool only once. The number of times that the JSQ$(d(N))$ scheme fails to assign a task to the `appropriate' server pool in any finite time interval, should be $\op(\sqrt{N})$. This gives rise to the additional $\log( N)$ factor in the growth rate of $d(N)$.

\subsection{Diffusion-limit results for integral \texorpdfstring{$\boldsymbol{\lambda}$}{lambda}}
We now turn to the case $f = 0$, and assume that
\begin{equation}
\label{eq:f=0}
\frac{K N - \lambda(N)}{\sqrt{N}} \to \beta \in \R \quad\mbox{ as }\quad N \to \infty,
\end{equation} 
which can be thought of as an analog of the so-called Halfin-Whitt
regime~\cite{HW81}.
As mentioned above, the limiting behavior in this case is
qualitatively different from the case $f > 0$.
Hence, we now consider the following scaled quantities:
\begin{equation}\label{eq:scaling-f=0}
\begin{split}
\hQ_{K-1}^{\sss d(N)}(t)&:=  \sum_{i=1}^{K-1}\frac{ N-Q_i^{\sss d(N)}(t)}{\sqrt{N}}\geq 0,\\
\hQ_K^{\sss d(N)}(t)&:=\frac{ N-Q_K^{\sss d(N)}(t)}{\sqrt{N}}\geq 0,\\
\hQ_i^{\sss d(N)}(t)&:=\frac{Q_i^{\sss d(N)}(t)}{\sqrt{N}}\geq 0,\quad \mathrm{for}\ i\geq K+1.
\end{split}
\end{equation}

\begin{theorem}
{\normalfont (Universality of diffusion limit for JSQ$(d(N))$ scheme, $f = 0$)}
\label{th:diff pwr of d 2}
Suppose there exists $M\geq K+1$, such that $Q^{\sss d(N)}_{M+1}(0)\equiv 0$, and 
$$(\hQ^{\sss d(N)}_{K-1}(0),\hQ^{\sss d(N)}_K(0),\ldots,\hQ^{\sss d(N)}_M(0))\dto (\hQ_{K-1}(0),\hQ_K(0),\ldots,\hQ_M(0))$$ in $\R^{\sss M-K+2}$. 
If $f = 0$, $d(N)/(\sqrt{N} \log(N))\to\infty$, Equation~\eqref{eq:f=0}
is satisfied, and $\hQ^{\sss d(N)}_{K-1}(0)\pto 0$, as $N\to\infty$, then the process $\left\{\big(\hQ^{\sss d(N)}_{K-1}(t),\hQ^{\sss d(N)}_K(t),\ldots,\hQ^{\sss d(N)}_M(t),\hQ^{\sss d(N)}_{M+1}(t)\big)\right\}_{t\geq 0}$ converges weakly to the process defined as the unique solution to the stochastic integral equation
\begin{equation}\label{eq:OU process2}
\begin{split}
\hQ_K(t) &= \hQ_K(0) + \sqrt{2K} W(t) -
\int_0^t (\hQ_K(s) + K \hQ_{K+1}(s)) \dif s + \beta t + V_1(t) \\
\hQ_{K+1}(t) &= \hQ_{K+1}(0) + V_1(t) - (K + 1) \int_0^t (\hQ_{K+1}(s)-\hQ_{K+2}(s)) \dif s,\\
\hQ_{i}(t) &= \hQ_{i}(0)  - i \int_0^t (\hQ_{i}(s)-\hQ_{i+1}(s)) \dif s, \quad i= K+2.\ldots, M-1,\\
\hQ_{M}(t) &= \hQ_{M}(0) - M \int_0^t \hQ_{M}(s) \dif s,
\end{split}
\end{equation}
$\hQ_{K-1}(t)\equiv 0$, and $\hQ_{M+1}(t)\equiv 0$, where $W(t)$ is the standard Brownian motion, and $V_1(t)$ is the unique
non-decreasing process in $D_{\R_+}[0,\infty)$ satisfying
$$\int_0^t \ind{\hQ_K(s) \geq 0} \dif V_1(s) = 0.$$
\end{theorem}
Unlike the $f > 0$ case, the above theorem says that, if $f = 0$,
then over any finite time horizon, there will be $O_P(\sqrt{N})$
server pools with fewer than $K$ or more than $K$~active tasks,
and hence most of the server pools have precisely $K$~active tasks.

\section{Proof Outline}\label{sec:eqiv}
The proofs of the asymptotic results for the JSQ$(d(N))$ scheme in Theorems~\ref{fluidjsqd}, \ref{th:diff pwr of d 1}, and \ref{th:diff pwr of d 2} involve two main components:
\begin{enumerate}[{\normalfont (i)}]
\item deriving the relevant limiting processes for the ordinary JSQ policy,
\item establishing a universality result which shows that the limiting processes for the JSQ$(d(N))$ scheme are `asymptotically equivalent' to those for the ordinary JSQ policy for suitably large $d(N)$.
\end{enumerate}
For Theorems~\ref{fluidjsqd}, \ref{th:diff pwr of d 1} and \ref{th:diff pwr of d 2}, part (i) will be dealt with in Theorems~\ref{th:genfluid},~\ref{th:diffusion} and~\ref{th: f=0 diffusion}, respectively. 
For all three theorems, part (ii) relies on a notion of asymptotic equivalence between different schemes, which is formalized in the next definition. 
\begin{definition}
Let $\Pi_1$ and $\Pi_2$ be two schemes parameterized by the number of server pools $N$. For any positive function $g:\N\to\R_+$, we say that $\Pi_1$ and $\Pi_2$ are `$g(N)$-alike' if there exists a common probability space, such that for any fixed $T\geq 0$, for all $i\geq 1$,
$$\sup_{t\in[0,T]}(g(N))^{-1}|Q_i^{\Pi_1}(t)-Q_i^{\Pi_2}(t)|\pto 0\quad \mathrm{as}\quad N\to\infty.$$
\end{definition}
Intuitively speaking, if two schemes are $g(N)$-alike, then in some sense, the associated system occupancy states are indistinguishable on $g(N)$-scale. 
For brevity, for two schemes $\Pi_1$ and $\Pi_2$ that are $g(N)$-alike, we will often say that $\Pi_1$ and $\Pi_2$ have the same process-level limits on $g(N)$-scale.
The next theorem states a sufficient criterion for the JSQ$(d(N))$ scheme and the ordinary JSQ policy to be $g(N)$-alike, and thus, provides the key vehicle in establishing the universality result in part (ii) mentioned above. 
\begin{theorem}\label{th:pwr of d}
Let $g:\N\to\R_+$ be a function diverging to infinity. Then the JSQ policy and the JSQ$(d(N))$ scheme are $g(N)$-alike, with $g(N)\leq N$, if 
\begin{align}\label{eq:fNalike cond1}
\mathrm{(i)}&\quad d(N)\to\infty, \quad\text{for}\quad g(N) = O(N),\\
\mathrm{(ii)}&\quad d(N)\left(\frac{N}{g(N)}\log\left(\frac{N}{g(N)}\right)\right)^{-1}\to\infty,\quad\text{for}\quad g(N)=o(N).\label{eq:fNalike cond2}
\end{align}
\end{theorem}
Theorem~\ref{th:pwr of d} can be intuitively explained as follows. 
The choice of $d(N)$ should be such that the JSQ$(d(N))$ scheme, at each arrival, with high probability selects one of the server pools with the minimum number of tasks, if the total number of server pools with the minimum number of tasks is of order $g(N)$. 
Moreover, in any finite time interval, the total number of times it fails to do so, should be of order lower than that of $g(N)$. 
These conditions imply that $d(N)$ must diverge if $g(N)=O(N)$, or grow faster than $(N/g(N))\log(N/g(N))$, if $g(N)=o(N)$. 

In order to obtain the fluid and diffusion limits for various schemes, the two main scales that we consider are $g(N)\sim N$ and $g(N)\sim\sqrt{N}$, respectively. 
The next two immediate corollaries of the above theorem will imply that it is enough to investigate the ordinary JSQ policy in various regimes.
\begin{corollary}\label{cor:fluid}
If $d(N)\to\infty$ as $N\to\infty$, then the JSQ$(d(N))$ scheme and the ordinary JSQ policy are $N$-alike.
\end{corollary}
\begin{remark}\label{rem:necessity-fluid}
{\normalfont
The growth condition on $d(N)$ in order for the JSQ$(d(N))$ scheme to be $N$-alike to the ordinary JSQ policy, stated in the above corollary, is not only sufficient, but also necessary.
Specifically, if $\liminf_{N\to\infty}d(N)\leq d<\infty$, then 
consider a subsequence along which the limit of $d(N)$ exists and is uniformly bounded by $d$.
Therefore, one can choose a further subsequence, such that $d(N)=d$ for all $N$ along that subsequence.
Now, from the fluid-limit result for the JSQ$(d)$ scheme~\cite{MKMG15, MMG15},
one can see that it differs from that of the JSQ policy stated in~\eqref{fluidjsqd}, and hence the JSQ(d(N)) scheme is not $N$-alike to the ordinary JSQ policy.
}
\end{remark}

\begin{corollary}\label{cor-diff}
If $d(N)/(\sqrt{N}\log (N))\to\infty$ as $N\to\infty$, then the JSQ$(d(N))$ scheme and the ordinary JSQ policy are $\sqrt{N}$-alike.
\end{corollary}

We will prove the universality result in Theorem~\ref{th:pwr of d} in the next section.
The key challenge is that
a direct comparison of the JSQ$(d(N))$ scheme and the ordinary JSQ policy is not straightforward. 
Hence, to compare the JSQ$(d(N))$ scheme with the JSQ policy, we adopt a two-stage approach based on a novel class of schemes, called CJSQ$(n(N))$, as a convenient intermediate scenario. 
Specifically, for some nonnegative integer-valued sequence $\big\{n(N)\big\}_{N\geq 1}$, with $n(N)\leq N$, we introduce a class of schemes named CJSQ($n(N)$), containing all the schemes that always assign the incoming task to one of the  $n(N)+1$ lowest ordered server pools.
Note that when $n(N)=0$, the class only contains the ordinary JSQ policy. 

Just like the JSQ$(d(N))$ scheme, the schemes in the class CJSQ$(n(N))$ may be thought of as ``sloppy'' versions of the JSQ policy, in the sense that tasks are not necessarily assigned to a server pool with the minimum number of active tasks but to one of the $n(N)+1$
lowest ordered server pools, as graphically illustrated in Figure~\ref{fig:sfig1}.
Below we often will not differentiate among the various schemes in the class CJSQ$(n(N))$, and prove a common property possessed by all these schemes. Hence, with minor abuse of notation, we will often denote a typical assignment scheme in this class by CJSQ($n(N)$).
Note that the JSQ$(d(N))$ scheme is guaranteed to identify the lowest ordered server pool, but only among a randomly sampled subset of $d(N)$ server pools.
In contrast, a scheme in the class in CJSQ$(n(N))$ only guarantees that
one of the $n(N)+1$ lowest ordered server pools is selected, but 
across the entire system of $N$ server pools. 
We will show that for sufficiently small $n(N)$, any scheme from the class CJSQ$(n(N))$ is still `close' to the ordinary JSQ policy in terms of $g(N)$-alikeness as stated in the next proposition.
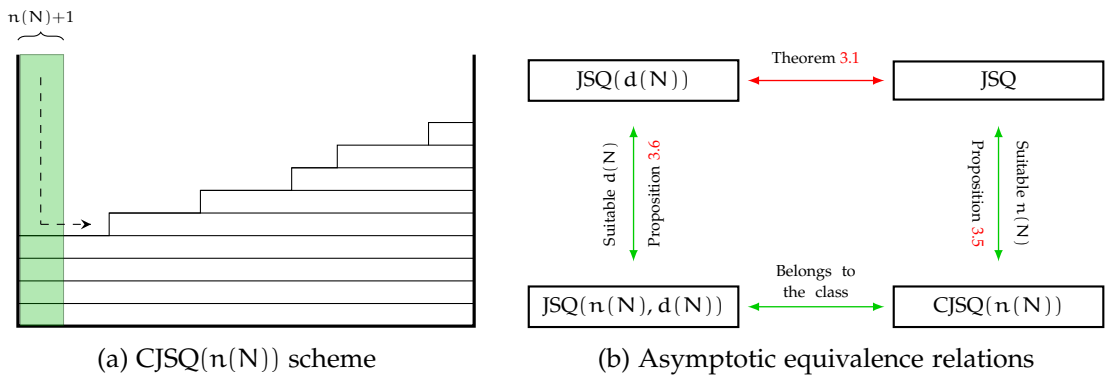
\begin{figure}
\begin{center}
\begin{subfigure}{.5\textwidth}
  \centering
  \begin{tikzpicture}[scale=.6]
\draw (1,0)--(1,2)--(3,2)--(3,2.5)--(5,2.5)--(5,3)--(7,3)--(7,3.5)--(8,3.5)--(8,4)--(10,4)--(10,4.5)--(11,4.5)--(11,0)--(1,0);

\draw  (1,.5)--(11,.5);
\draw  (1,1)--(11,1);
\draw  (1,1.5)--(11,1.5);
\draw  (1,2)--(11,2);
\draw  (3,2.5)--(11,2.5);
\draw  (5,3)--(11,3);
\draw  (7,3.5)--(11,3.5);
\draw  (10,4)--(11,4);

\draw[very thick] (1,6)--(1,0)--(11,0)--(11,6);

\draw[fill=green!70!black,opacity=0.3] (1.05,0) rectangle (2,6);
\draw[dashed] (1.5,5.5) -- (1.5, 2.25);
\draw[dashed, decoration={markings,mark=at position 1 with
    {\arrow[scale=1.2,>=stealth]{>}}},postaction={decorate}] (1.5, 2.25) -- (2.6, 2.25);

\draw [decorate,decoration={brace,amplitude=3pt},xshift=0pt,yshift=2pt]
(1,6.25) -- (2,6.25) node [black,midway,yshift=0.3cm] { $\scriptscriptstyle n(N)+1$};

\end{tikzpicture}
  \caption{CJSQ$(n(N))$ scheme\vspace{16pt}}
  \label{fig:sfig1}
\end{subfigure}%
\begin{subfigure}{.5\textwidth}
  \centering
  \begin{tikzpicture}[scale=.6]
  
  \draw(14,1) node[mybox, text width = 2.5cm,text centered]  {%
   {\scriptsize JSQ$(n(N),d(N))$}};
   
  \draw(22,1) node[mybox, text width = 2.5cm,text centered]  {%
   {\scriptsize  CJSQ$(n(N))$}};
   
   \draw(14,6) node[mybox, text width = 2.5cm,text centered]  {%
   {\scriptsize  JSQ$(d(N))$}};
   
   \draw(22,6) node[mybox, text width = 2.5cm,text centered]  {%
   {\scriptsize  JSQ}};

\draw[doublearr] (16.5,6) to (19.5,6);
\draw[doublearr2] (14,2) to (14,5);
\draw[doublearr2] (16.5,1) to (19.5,1);
\draw[doublearr2] (22,2) to (22,5);

\node  at (18,6.5) {\tiny Theorem~\ref{th:pwr of d}};

\node [rotate=270] at (21.5,3.5) {\tiny Proposition~\ref{prop: modified JSL}};
\node [rotate=270] at (22.5,3.5) {\tiny Suitable $n(N)$};

\node [rotate=90] at (14.5,3.5) {\tiny Proposition~\ref{prop: power of d}};
\node [rotate=90] at (13.5,3.5) {\tiny Suitable $d(N)$};

\draw[text width=1.1cm, align=center] (18,1.75) node {\tiny Belongs to};
\draw[text width=1.1cm, align=center] (18,1.35) node {\tiny the class};
\end{tikzpicture}
  \caption{Asymptotic equivalence relations}
  \label{fig:sfig3}
\end{subfigure}
\caption{(Left) The class CJSQ$(n(N))$ is depicted in a high-level view of the system, where as in Figure~\ref{fig:1} the server pools are arranged in nondecreasing order of the number of active tasks, and the arrival must be assigned through the left tunnel. (Right) The equivalence structure is depicted for various intermediate load balancing schemes to facilitate the comparison between the JSQ$(d(N))$ scheme and the ordinary JSQ policy.
}
\label{fig:strategy}
\end{center}
\end{figure}
\begin{proposition}\label{prop: modified JSL}
For any function $g:\N\to\R_+$ diverging to infinity, if $n(N)/ g(N)\to 0$  as $N\to\infty$, then the JSQ policy and the CJSQ$(n(N))$ schemes are $g(N)$-alike.
\end{proposition}
In order to prove this proposition, we introduce in Section~\ref{sec:stoch-coupling} a novel stochastic coupling called the T-coupling, to construct a common probability space, and establish the property of $g(N)$-alikeness.\\

Next we compare the CJSQ($n(N)$) schemes with the JSQ($d(N)$) scheme. 
The comparison follows a somewhat similar line of argument as 
in \cite[Section~4]{MBLW16-3}, and involves a JSQ$(n(N),d(N))$ scheme
which is an intermediate blend between the CJSQ$(n(N))$ schemes and the
JSQ$(d(N))$ scheme.
Specifically, the JSQ$(n(N),d(N))$ scheme selects a candidate server pool in the exact same way as the JSQ$(d(N))$ scheme.
However, it only assigns the task to that server pool if it belongs to 
the $n(N)+1$ lowest ordered ones,
and to a randomly selected server pool among these otherwise.
By construction, the JSQ$(n(N),d(N))$ scheme belongs to the class CJSQ$(n(N))$.

We now consider two T-coupled systems with a JSQ$(d(N))$
and a JSQ$(n(N),d(N))$ scheme.
Assume that at some specific arrival epoch, the incoming task is assigned to the $k^{\mathrm{th}}$ ordered server pool in the system under the JSQ($d(N)$) scheme. 
If $k\in\big\{1,2,\ldots,n(N)+1\big\}$, then the scheme JSQ$(n(N),d(N))$ also assigns the arriving task to the $k^{\mathrm{th}}$ ordered server pool. 
Otherwise it dispatches the arriving task uniformly at random among the first $n(N)+1$ ordered server pools.

We will establish a sufficient criterion on $d(N)$ in order for the JSQ$(d(N))$ scheme and JSQ$(n(N),d(N))$ scheme to be close in terms of $g(N)$-alikeness, as stated in the next proposition.
\begin{proposition}\label{prop: power of d}
Assume, $n(N)/g(N)\to 0$ as $N\to\infty$ for some function $g:\N\to\R_+$ diverging to infinity. The JSQ$(n(N),d(N))$ scheme and the JSQ($d(N)$) scheme are $g(N)$-alike if the following condition holds:
\begin{equation}\label{eq:condition-same}
\frac{n(N)}{N}d(N)-\log\frac{N}{g(N)}\to\infty, \quad\text{as}\quad N\to\infty.
\end{equation}
\end{proposition} 
Finally, Proposition~\ref{prop: power of d} in conjunction with Proposition~\ref{prop: modified JSL} yields Theorem~\ref{th:pwr of d}.
The overall proof strategy as described above, is schematically represented in Figure~\ref{fig:sfig3}.

\begin{remark} \normalfont
Note that, sampling \emph{without} replacement polls more server pools than \emph{with} replacement, and hence the minimum number of active tasks among the selected server pools is stochastically smaller in the case without replacement.
As a result, for sufficient conditions as in Theorem~\ref{th:pwr of d} it is enough to consider sampling with replacement.
\end{remark}

\section{Universality Property}
\label{sec:proof-equiv}

In this section we formalize the proof outlined in the previous section.
In Subsection~\ref{sec:stoch-coupling} we first introduce the T-coupling between any two task assignment schemes.
This coupling is used to derive stochastic inequalities in Subsection~\ref{ssec:stoch-ineq}, stated as Proposition~\ref{prop:stoch-ord2} and Lemma~\ref{lem:majorization}, which
in turn,  are used to prove Propositions~\ref{prop: modified JSL},~\ref{prop: power of d} and Theorem~\ref{th:pwr of d} in Subsection~\ref{ssec:asymp-eq}.
\subsection{Stochastic coupling}\label{sec:stoch-coupling}
Throughout this subsection we fix $N$, and suppress the superscript $N$ in the notation. Let $Q_{i}^{\Pi_1}(t)$ and $Q_{i}^{\Pi_2}(t)$ denote the number of server pools with at least $i$ active tasks, at time $t$, in two systems following schemes $\Pi_1$ and $\Pi_2$, respectively.
With a slight abuse of terminology, we occasionally use $\Pi_1$ and $\Pi_2$ to refer to systems following schemes $\Pi_1$ and $\Pi_2$, respectively.
To couple the two systems, we synchronize the arrival epochs 
and maintain a single exponential departure clock with instantaneous rate at time $t$ given by $M(t):=\max\left\{\sum_{i=1}^BQ_{i}^{\Pi_1}(t), \sum_{i=1}^BQ_{i}^{\Pi_2}(t)\right\}$.
 We couple the arrivals and departures in the various server pools as follows:

(1) \emph{Arrival:} At each arrival epoch, assign the incoming task in each system to one of the server pools according to the respective schemes.

(2) \emph{Departure:} 
Define
$$H(t):=\sum_{i=1}^{B}\min\left\{Q_{i}^{\Pi_1}(t), Q_{i}^{\Pi_2}(t)\right\}$$
and
$$p(t):=
\begin{cases}
\dfrac{H(t)}{M(t)},&\quad\text{if}\quad M(t)>0,\\
0,&\quad \text{otherwise.}
\end{cases}
$$
At each departure epoch $t_k$ (say), draw a uniform$[0,1]$ random variable $U(t_k)$. 
The departures occur in a coupled way based upon the value of $U(t_k)$. In either of the systems, assign a task index $(i,j)$, if that task is at the $j^{\mathrm{th}}$ position of the $i^{\mathrm{th}}$ ordered server pool. 
Let $\mathcal{A}_1(t)$ and $\mathcal{A}_2(t)$ denote the set of all task-indices present at time $t$ in systems $\Pi_1$ and $\Pi_2$, respectively. 
Color the indices (or tasks) in $\mathcal{A}_1\cap\mathcal{A}_2$, $\mathcal{A}_1\setminus\mathcal{A}_2$ and $\mathcal{A}_2\setminus\mathcal{A}_1$, green, blue and red, respectively,
and note that $|\mathcal{A}_1\cap \mathcal{A}_2|=H(t)$. 
Define a total order on the set of indices as follows:
$(i_1,j_1)<(i_2,j_2)$ if $i_1<i_2$, or $i_1=i_2$ and $j_1<j_2$.  Now, if $U(t_k)\leq p(t_k-)$, then select one green index uniformly at random and remove the corresponding tasks from both systems. Otherwise, if $U(t_k)> p(t_k-)$, then choose one integer $m$, 
uniformly at random from all the integers between 1 and  $M(t)-H(t)=M(t)(1-p(t))$, and remove the tasks corresponding to the $m^{\mathrm{th}}$ smallest (according to the order defined above) red and blue indices in the corresponding systems.
If the number of red (or blue) tasks is less than $m$, then do nothing.

\begin{figure}\label{fig:2}
\begin{center}
\begin{tikzpicture}[scale=.70]
\foreach \x in {10, 9,...,1}
	\draw (\x,6)--(\x,0)--(\x+.7,0)--(\x+.7,6);
\foreach \x in {10, 9,...,1}
	\draw (\x+.35,-.45) node[circle,inner sep=0pt, minimum size=10pt,draw,thick, teal] {{{\tiny $\mathsmaller{\x}$}}} ;
\foreach \y in {0, .5}
	\draw[fill=BlueViolet,BlueViolet] (1.15,.1+\y) rectangle (1.55,.5+\y);
	
\foreach \y in {0, .5}
	\draw[fill=green!70!black,green!70!black] (2.15,.1+\y) rectangle (2.55,.5+\y);
\foreach \y in {1, 1.5}
	\draw[fill=BrickRed,BrickRed] (2.15,.1+\y) rectangle (2.55,.5+\y);

\foreach \y in {0, .5, 1, 1.5}
	\draw[fill=green!70!black,green!70!black] (3.15,.1+\y) rectangle (3.55,.5+\y);

\foreach \y in {0, .5, 1, 1.5, 2}
	\draw[fill=green!70!black,green!70!black] (4.15,.1+\y) rectangle (4.55,.5+\y);
\foreach \y in {2.5}
	\draw[fill=BlueViolet,BlueViolet] (4.15,.1+\y) rectangle (4.55,.5+\y);

\foreach \y in {0, .5, 1, 1.5, 2}
	\draw[fill=green!70!black,green!70!black] (5.15,.1+\y) rectangle (5.55,.5+\y);
\foreach \y in {2.5, 3}
	\draw[fill=BlueViolet,BlueViolet] (5.15,.1+\y) rectangle (5.55,.5+\y);

\foreach \y in {0, .5, 1, 1.5, 2, 2.5, 3, 3.5, 4}
	\draw[fill=green!70!black,green!70!black] (6.15,.1+\y) rectangle (6.55,.5+\y);
\foreach \y in {0, .5, 1, 1.5, 2, 2.5, 3, 3.5, 4}
	\draw[fill=green!70!black,green!70!black] (7.15,.1+\y) rectangle (7.55,.5+\y);

\foreach \y in {0, .5, 1, 1.5, 2, 2.5, 3, 3.5, 4, 4.5}
	\draw[fill=green!70!black,green!70!black] (8.15,.1+\y) rectangle (8.55,.5+\y);
\foreach \y in {5}
	\draw[fill=BrickRed,BrickRed] (8.15,.1+\y) rectangle (8.55,.5+\y);

\foreach \y in {0, .5, 1, 1.5, 2, 2.5, 3, 3.5, 4, 4.5}
	\draw[fill=green!70!black,green!70!black] (9.15,.1+\y) rectangle (9.55,.5+\y);
\foreach \y in { 5}
	\draw[fill=BrickRed,BrickRed] (9.15,.1+\y) rectangle (9.55,.5+\y);	
	
\foreach \y in {0, .5, 1, 1.5, 2, 2.5, 3, 3.5, 4, 4.5}
	\draw[fill=green!70!black,green!70!black] (10.15,.1+\y) rectangle (10.55,.5+\y);
\foreach \y in {5}
	\draw[fill=BrickRed,BrickRed] (10.15,.1+\y) rectangle (10.55,.5+\y);

\draw (1.35,.3) node[white] {\scriptsize 1};
\draw (1.35,.8) node[white] {\scriptsize 2};
\draw (2.35,1.3) node[white] {\scriptsize 1};
\draw (2.35,1.8) node[white] {\scriptsize 2};
\draw (4.35,2.8) node[white] {\scriptsize 3};
\draw (5.35,2.8) node[white] {\scriptsize 4};
\draw (5.35,3.3) node[white] {\scriptsize 5};
\draw (8.35,5.3) node[white] {\scriptsize 3};
\draw (9.35,5.3) node[white] {\scriptsize 4};
\draw (10.35,5.3) node[white] {\scriptsize 5};

\end{tikzpicture}
\end{center}
\caption{Superposition of the occupancy states at sgreenome particular time instant, of  schemes $\Pi_1$ and $\Pi_2$ when the server pools in both systems are arranged in nondecreasing order of the number of active tasks. The $\Pi_1$ system is the union of the green and blue tasks, and the $\Pi_2$ system is the union of the green and red tasks.}
\label{fig:modified jsl}
\end{figure}
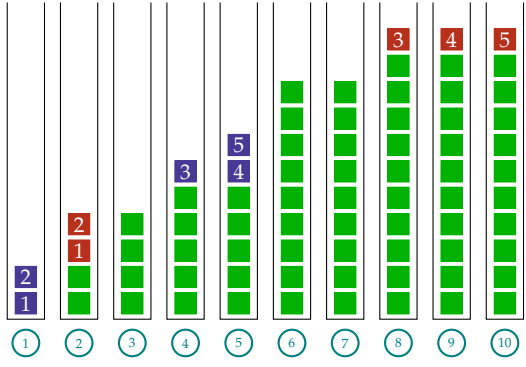

The above coupling has been schematically represented in Figure~\ref{fig:modified jsl},
and will henceforth be referred to as T-coupling, where T stands for `task-based'. 
Now we need to show that, under the T-coupling, the two systems, considered independently, evolve according to their own statistical laws. 
This can be seen in several steps.
Indeed, the T-coupling basically uniformizes the departure rate by the maximum number of tasks present in either of the two systems. 
Then informally speaking, the green regions signifies the common portion of tasks, and the red and  blue region represent the separate contributions. 
Now observe that
\begin{enumerate}[{\normalfont (i)}]
\item The total departure rate from $\Pi_i$ is 
\begin{align*}
&M(t)\left[p(t)+(1-p(t))\frac{|\mathcal{A}_i\setminus\mathcal{A}_{3-i}|}{M(t)-H(t)}\right]
= |\mathcal{A}_1\cap\mathcal{A}_2|+|\mathcal{A}_i\setminus\mathcal{A}_{3-i}|
=|\mathcal{A}_i|, \quad i= 1,2.
\end{align*}
\item Assuming without loss of generality~$|\mathcal{A}_1|\geq |\mathcal{A}_2|$, each task in $\Pi_1$ is equally likely to depart.
\item Each task in $\Pi_2$ within $\mathcal{A}_1\cap\mathcal{A}_2$ and 
each task within $\mathcal{A}_2\setminus\mathcal{A}_1$ is equally likely to depart, and the probabilities of departures are proportional to $|\mathcal{A}_1\cap\mathcal{A}_2|$ and $|\mathcal{A}_2\setminus\mathcal{A}_1|$, respectively.
\end{enumerate}

\subsection{Stochastic inequalities}\label{ssec:stoch-ineq}

Now, as in \cite{MBLW16-3} we define a notion of comparison between two T-coupled systems. 
Two T-coupled systems are said to~\emph{differ in decision} at some arrival epoch, if the index of the ordered server pool joined by the arriving task at that epoch, differs in the two systems.
Denote by $\Delta_{\Pi_1,\Pi_2}(t)$, the cumulative number of times that the two systems~$\Pi_1$ and~$\Pi_2$ differ in decision up to time $t$. 

\begin{proposition}\label{prop:stoch-ord2}
For two T-coupled systems under any two schemes $\Pi_1$ and $\Pi_2$ the following inequality is preserved
\begin{equation}\label{eq:stoch-ord2}
\sum_{i=1}^B \big|Q_i^{\Pi_1}(t)-Q_i^{\Pi_2}(t)\big|\leq 2\Delta_{\Pi_1,\Pi_2}(t)\qquad\forall\ t\geq 0,
\end{equation} 
provided the two systems start from the same occupancy state at time $t=0$.
\end{proposition}
The proof follows a somewhat similar line of argument as in \cite{MBLW16-3, MBLW15}, but is provided below since the coupling is different here. 
For any scheme $\Pi$, define $I_\Pi(c):=\max\big\{i:Q_i^\Pi\geq N-c+1\big\}$, $c=1,\ldots,N$.
\begin{proof}[Proof of Proposition~\ref{prop:stoch-ord2}]
We use forward induction on event times, i.e., time epochs when either an arrival or a departure takes place. 
Assume the inequality in~\eqref{eq:stoch-ord2} holds at time epoch $t_0$. 
We 
denote by $\tilde{Q}^{\Pi}$ the updated occupancy state after the next event at time epoch $t_1$, and
 distinguish between two cases depending on whether $t_1$ is an arrival epoch or a departure epoch.

If $t_1$ is an arrival epoch and
 if the systems differ in decision, then observe that the left side of \eqref{eq:stoch-ord2} can increase at most by two. In this case, the right side also increases by two, and the ordering is preserved.
Therefore, it is enough to prove that the right side of \eqref{eq:stoch-ord2} remains unchanged if the two systems do not differ in decision. In that case, assume that both $\Pi_1$ and $\Pi_2$ assign the arriving task to the $k^{\mathrm{th}}$ ordered server pool.
Then
\begin{equation}\label{eq:addition}
\tilde{Q}^{\Pi}_i=
\begin{cases}
Q^{\Pi}_i+1, &\mbox{ for }i=I_{\Pi}(k)+1,\\
Q^{\Pi}_j,&\mbox{ otherwise, }
\end{cases}
\end{equation}
if $I_{\Pi}(k)<B$; otherwise all the $Q_i$-values remain unchanged. 
If $I_{\sss\Pi_1}(k)=I_{\sss\Pi_2}(k)$, then the left side of \eqref{eq:stoch-ord2} clearly remains unchanged. Now, without loss of generality, assume $I_{\sss\Pi_1}(k)<I_{\sss\Pi_2}(k)$. Therefore, 
$$Q_{ I_{\Pi_1}(k)+1}^{\Pi_1}(t_0)< Q_{ I_{\Pi_1}(k)+1}^{\Pi_2}(t_0)\quad \mathrm{and}\quad Q_{ I_{\Pi_2}(k)+1}^{\Pi_1}(t_0)<Q_{ I_{\Pi_2}(k)+1}^{\Pi_2}(t_0).$$
After an arrival, the $ (I_{\Pi_1}(k)+1)^{\mathrm{th}}$ term in the left side of \eqref{eq:stoch-ord2} decreases by one, and the $ (I_{\Pi_2}(k)+1)^{\mathrm{th}}$ term increases by one. Thus the inequality is preserved.

If $t_1$ is a departure epoch, then first consider the case when the departure occurs from the green region. In that case, without loss of generality, assume that a potential departure occurs from the $k^{\mathrm{th}}$ ordered server pool, for some $k\in\big\{1,2,\ldots,N\big\}.$ Also note that a departure in either of the two systems can change at most one of the $Q_i$-values.
Thus
\begin{equation}\label{eq:removal}
\tilde{Q}^{\Pi}_i=
\begin{cases}
Q^{\Pi}_i-1, &\mbox{ for }i=I_{\Pi}(k),\\
Q^{\Pi}_j,&\mbox{ otherwise, }
\end{cases}
\end{equation}
if $I_{\Pi}(k)\geq 1$; otherwise all the $Q_i$-values remain unchanged.

If at time epoch $t_0$, $I_{\Pi_1}(k)=I_{\Pi_2}(k)=I$, then both $Q_{\sss I^{\Pi_1}}$ and $Q_{\sss I^{\Pi_2}}$ decrease by one, and hence the left side of \eqref{eq:stoch-ord2} does not change. 

Otherwise, without loss of generality assume $I_{\Pi_1}(k)<I_{\Pi_2}(k).$ Then observe that 
$$Q_{ I_{\Pi_1}(k)}^{\Pi_1}(t_0)\leq Q_{ I_{\Pi_1}(k)}^{\Pi_2}(t_0)\quad \mathrm{and}\quad Q_{ I_{\Pi_2}(k)}^{\Pi_1}(t_0)<Q_{ I_{\Pi_2}(k)}^{\Pi_2}(t_0).$$ 
Furthermore, after the departure, $Q_{\sss I_{\Pi_1}(k)}^{\Pi_1}$ decreases by one, therefore $|Q_{\sss I_{\Pi_1}(k)}^{\Pi_1}- Q_{\sss I_{\Pi_1}(k)}^{\Pi_2}|$ increases by one, and $Q_{\sss I_{\Pi_2}(k)}^{\Pi_2}$ decreases by one, thus $|Q_{\sss I_{\Pi_2}(k)}^{\Pi_1}- Q_{I_{\sss \Pi_2}(k)}^{\Pi_2}|$ decreases by one. Hence, in total, the left side of \eqref{eq:stoch-ord2} remains the same.
Now if a departure occurs from the blue and/or red region, then for some $i_1$ and/or $i_2$, $(Q_{i_1}^{\Pi_1}-Q_{i_1}^{\Pi_2})^+$ or $(Q_{i_2}^{\Pi_2}-Q_{i_2}^{\Pi_1})^+$ (or both) decreases, and the other terms remain unchanged, and hence the left side clearly decreases or remains unchanged.
\end{proof}

In order to compare the JSQ policy with the CJSQ(n(N)) schemes,
denote by $Q_i^{\Pi_1}(t)$ and $Q_i^{\Pi_2}(t)$ the number of server pools with at least $i$ tasks under the JSQ policy and CJSQ$(n(N))$ scheme, respectively.
 Now, in order to prove Proposition~\ref{prop: modified JSL}, we will need the following lemma.

\begin{lemma}\label{lem:majorization}
For any $k\in\big\{1,2,\ldots B\big\}$, 
\begin{equation}\label{eq:majorization}
\left\{\sum_{i=1}^k Q_i^{\Pi_1}(t)-kn(N)\right\}_{t\geq 0}\leq_{st} \left\{\sum_{i=1}^k Q_i^{\Pi_2}(t)\right\}_{t\geq 0}\leq_{st}\left\{\sum_{i=1}^k Q_i^{\Pi_1}(t)\right\}_{t\geq 0}, 
\end{equation}
provided at $t=0$ the two systems start from the same occupancy states.
\end{lemma}
In the next two remarks we comment on the contrast of Lemma~\ref{lem:majorization} and the underlying T-coupling with stochastic dominance properties for the ordinary JSQ policy in the existing literature and the S-coupling technique in reference~\cite{MBLW16-3}, respectively
\begin{remark}\label{rem:novelty}
{\normalfont
The stochastic ordering in Lemma~\ref{lem:majorization} is to be contrasted with the weak majorization results in~\cite{Winston77, towsley, Towsley95, Towsley1992, W78} in the context of the ordinary JSQ policy in the single-server queueing scenario, and in~\cite{STC93,J89, M87, MS91} in the scenario of state-dependent service rates, non-decreasing with the number of active tasks.
In the current infinite-server scenario, the results in~\cite{STC93,J89, M87, MS91} imply that for any non-anticipating scheme~$\Pi$ taking assignment decision based on the number of active tasks only, for all $t\geq 0$,
\begin{align}\label{eq: towsley}
\sum_{m=1}^\ell X_{(m)}^{JSQ}(t)&\leq_{st}\sum_{m=1}^\ell X_{(m)}^{\Pi}(t),\mbox{ for } \ell=1,2,\ldots, N,\\
\left\{L^{JSQ}(t)\right\}_{t\geq 0}&\leq_{st}\left\{L^{\Pi}(t)\right\}_{t\geq 0},
\end{align}
where $X^{\Pi}_{(m)}(t)$ is the number of tasks in the $m^{\mathrm{th}}$ ordered server pool at time $t$ in the system following scheme $\Pi$ and $L^{\Pi}(t)$ is the total number of overflow events under policy $\Pi$ up to time $t$. 
Observe that $X_{(m)}^{\Pi}$ can be visualized as the $m^{\mathrm{th}}$ largest (rightmost) vertical bar (or stack) in Figure~\ref{fig:1}.
 Thus~\eqref{eq: towsley} says that the sum of the lengths of the $\ell$ largest \emph{vertical} stacks in a system following any scheme $\Pi$ is stochastically larger than or equal to that following the ordinary JSQ policy for any $\ell=1,2,\ldots,N$. Mathematically, this ordering can be equivalently written as
 \begin{equation}\label{eq:equiv-ord}
 \sum_{i = 1}^{B} \min\big\{\ell, Q_i^{JSQ}(t)\big\}  \leq_{st}
\sum_{i = 1}^{B} \min\big\{\ell, Q_i^{\Pi}(t)\big\},
 \end{equation}
for all $\ell = 1, \dots, N$.
In contrast, in order to show asymptotic equivalence on various scales, we need to both upper and lower bound the occupancy states of the CJSQ$(n(N))$ schemes in terms of the JSQ policy, and therefore need a much stronger hold on the departure process.
The T-coupling provides us just that, and has several useful properties that are crucial for our proof technique.
For example, Proposition~\ref{prop:stoch-ord2} uses the fact that if two systems are T-coupled, then departures cannot increase the sum of the absolute differences of the $Q_i$-values, which is not true for the coupling considered in the above-mentioned  literature.
The left stochastic ordering in~\eqref{eq:majorization} also does not remain valid in those cases.
Furthermore, observe that the right inequality in~\eqref{eq:majorization} (i.e., $Q_i$'s) implies the stochastic inequality is \emph{reversed} in~\eqref{eq:equiv-ord}, which is counter-intuitive in view of the optimality properties of the ordinary JSQ policy studied in the literature, as mentioned above.
The fundamental distinction between the two coupling techniques is also reflected by the fact that the T-coupling does not allow for arbitrary nondecreasing state-dependent departure rate functions, unlike the couplings in~\cite{STC93,J89, M87, MS91}.}
\end{remark}

\begin{remark}\label{rem:contrast}\normalfont
As briefly mentioned in the introduction, in the current infinite-server scenario, the departures of the ordered server pools cannot be coupled, mainly since the departure rate at the $m^{\rm th}$ ordered server pool, for some $m = 1,2,\ldots, N$, depends on its number of active tasks.
It is worthwhile to mention that the coupling in this paper is stronger than that used in~\cite{MBLW16-3}.
Observe that due to Lemma~\ref{lem:majorization}, the absolute difference of the occupancy states of the JSQ policy and any scheme from the CJSQ class at any time point can be bounded deterministically (without any terms involving the cumulative number of lost tasks).
It is worth emphasizing that the universality result on some specific scale, stated in Theorem~\ref{th:pwr of d} does not depend on the behavior of the JSQ policy on that scale, whereas in~\cite{MBLW16-3} it does, mainly because the upper and lower bounds in \cite[Corollary 3.3]{MBLW16-3} involve tail sums of two different policies.
Also, the bound in the current paper does not depend upon $t$, and hence, applies in the steady state as well.
Moreover, the coupling in~\cite{MBLW16-3} compares the $k$ \emph{highest} horizontal bars, whereas the present paper compares the $k$ \emph{lowest} horizontal bars.
As a result, the bounds on the occupancy states established in~\cite[Corollary 3.3]{MBLW16-3} involves tail sums of the occupancy states of the ordinary JSQ policy,
which necessitates proving the $\ell_1$ convergence of the occupancy states of the ordinary JSQ policy. 
In contrast, the bound we establish in the present paper, involves only a single component (see equations \eqref{eq:upperocc} and \eqref{eq:lowerocc}), and thus, the convergence with respect to product topology suffices.
\end{remark}

\begin{proof}[Proof of Lemma~\ref{lem:majorization}]
Fix any $k\geq 1$. We will use forward induction on the event times, i.e., time epochs when either an arrival or a departure occurs,
and assume the two systems to be T-coupled as described in Section~\ref{sec:stoch-coupling}. 
We suppose that the two inequalities hold at time epoch $t_0$, and will prove that they continue to hold at time epoch $t_1$.

(a) We first prove the left inequality in~\eqref{eq:majorization}. 
We distinguish between two cases depending on whether the next event time $t_1$ is an arrival epoch or a departure epoch.
We first consider the case of an arrival.
Since at each arrival, there can be an increment of size at most one, if $\sum_{i=1}^k Q_i^{\Pi_1}(t_0)-kn(N)< \sum_{i=1}^k Q_i^{\Pi_2}(t_0)$, the inequality holds trivially at time $t_1$. Therefore, consider the case when $\sum_{i=1}^k Q_i^{\Pi_1}(t_0)-kn(N)= \sum_{i=1}^k Q_i^{\Pi_2}(t_0)$. Now observe that,
$$\sum_{i=1}^k Q_i^{\Pi_2}(t_0)=\sum_{i=1}^k Q_i^{\Pi_1}(t_0)-kn(N)\leq kN-kn(N).$$
Hence, $Q_k^{\Pi_2}(t_0)\leq N-n(N)$, which in turn implies that at time $t_1$, $\sum_{i=1}^k Q_i^{\Pi_2}$ increases by 1, and the inequality is preserved. 
We now assume the case of a departure.
Then also if $\sum_{i=1}^k Q_i^{\Pi_1}(t_0)-kn(N)< \sum_{i=1}^k Q_i^{\Pi_2}(t_0)$, the inequality holds trivially at time $t_1$. Otherwise assume $\sum_{i=1}^k Q_i^{\Pi_1}(t_0)-kn(N)= \sum_{i=1}^k Q_i^{\Pi_2}(t_0)$. 
In this case if the departure occurs from the green region in Figure~\ref{fig:modified jsl}, then both $\sum_{i=1}^k Q_i^{\Pi_1}$ and $\sum_{i=1}^k Q_i^{\Pi_2}$ change in a similar fashion (i.e., either decrease by one or remain unchanged). 
Else, if the departure occurs from the red and blue regions, since $\sum_{i=1}^k Q_i^{\Pi_1}\geq \sum_{i=1}^k Q_i^{\Pi_2}$, by virtue of the T-coupling, if $\sum_{i=1}^k Q_i^{\Pi_2}$ decreases by one, then so does $\sum_{i=1}^k Q_i^{\Pi_1}$. 
To see this observe the following:
\begin{equation}
\sum_{i=1}^k Q_i^{\Pi_1}\geq \sum_{i=1}^k Q_i^{\Pi_2}\implies\sum_{i=1}^{k}(Q_i^{\Pi_1}-Q_i^{\Pi_2})^+\geq\sum_{i=1}^{k}(Q_i^{\Pi_2}-Q_i^{\Pi_1})^+.
\end{equation}
Therefore, if $m\leq \sum_{i=1}^{k}(Q_i^{\Pi_2}-Q_i^{\Pi_1})^+$, then $m\leq \sum_{i=1}^{k}(Q_i^{\Pi_1}-Q_i^{\Pi_2})^+$.
 Hence the inequality will be preserved.

(b) We now prove the right inequality in~\eqref{eq:majorization}
and again distinguish between two cases. 
If $t_1$ is an arrival epoch, 
 we assume for a similar reason as above, $\sum_{i=1}^k Q_i^{\Pi_2}(t_0)=\sum_{i=1}^k Q_i^{\Pi_1}(t_0)$. 
In this case when a task arrives, if it gets admitted under the CJSQ($n(N)$) scheme and increases $\sum_{i=1}^k Q_i^{\Pi_2}$, then clearly $\sum_{i=1}^k (N-Q_i^{\Pi_1}(t))>0$, and hence the incoming task will increase $\sum_{i=1}^k Q_i^{\Pi_1}$, as well, and the inequality will be preserved. 
If $t_1$ is a departure epoch with $\sum_{i=1}^k Q_i^{\Pi_2}(t_0)=\sum_{i=1}^k Q_i^{\Pi_1}(t_0)$, then by virtue of the T-coupling again, if  $\sum_{i=1}^k Q_i^{\Pi_1}$ decreases by one, then by the argument in (a) above, so does $\sum_{i=1}^k Q_i^{\Pi_2}$, thus preserving the inequality.
\end{proof}

\subsection{Asymptotic equivalence}\label{ssec:asymp-eq}

\begin{proof}[Proof of Proposition~\ref{prop: modified JSL}]
Using Lemma~\ref{lem:majorization},
there exists a common probability space such that for any $k\geq 1$ we can write
\begin{equation}\label{eq:upperocc}
\begin{split}
Q_k^{\Pi_2}(t)&=\sum_{i=1}^k Q_i^{\Pi_2}(t)-\sum_{i=1}^{k-1} Q_i^{\Pi_2}(t)\\
&\leq \sum_{i=1}^{k} Q_i^{\Pi_1}(t)-\sum_{i=1}^{k-1} Q_i^{\Pi_1}(t)+kn(N)\\
&=Q_k^{\Pi_1}(t)+kn(N).
\end{split}
\end{equation}
Similarly, we can write
\begin{equation}\label{eq:lowerocc}
\begin{split}
Q_k^{\Pi_2}(t)&=\sum_{i=1}^k Q_i^{\Pi_2}(t)-\sum_{i=1}^{k-1} Q_i^{\Pi_2}(t)\\
&\geq \sum_{i=1}^{k} Q_i^{\Pi_1}(t)-kn(N)-\sum_{i=1}^{k-1} Q_i^{\Pi_1}(t)\\
&=Q_k^{\Pi_1}(t)-kn(N).
\end{split}
\end{equation}
Therefore, for all $k\geq 1$, we have, $\sup_t|Q_k^{\Pi_2}(t)-Q_k^{\Pi_1}(t)|\leq kn(N)$. Since $n(N)/ g(N)\to\infty$ as $N\to\infty$, the proof is complete.
\end{proof}


\begin{proof}[Proof of Proposition~\ref{prop: power of d}]
For any $T\geq 0$, let $A^N(T)$ and $\Delta^N(T)$ be the total number
of arrivals to the system and the cumulative number of times that the
JSQ$(d(N))$ scheme and the JSQ$(n(N),d(N))$ scheme differ in decision up to time $T$.
Using Proposition~\ref{prop:stoch-ord2} it suffices to show that
for any $T\geq 0$, $\Delta^N(T)/g(N)\pto 0$ as $N\to\infty$.
Observe that at any arrival epoch, the systems under the JSQ$(d(N))$ and
JSQ$(n(N),d(N))$ schemes will differ in decision only if none of the $n(N)+1$ lowest ordered server pools get selected by the JSQ$(d(N))$ scheme.

Now at the time of an arrival, the probability that the JSQ$(d(N))$ scheme does not select one of the $n(N)+1$ lowest ordered server pools, is given by 
$$p(N)=\left(1-\frac{n(N)+1}{N}\right)^{\sss d(N)}.$$
Since at each arrival epoch $d(N)$ server pools are selected independently, given $A^N(T)$, $\Delta^N(T)\sim \mbox{Bin}(A_N(T),p(N))$.

Note that, for $T\geq 0$, Markov's inequality yields
$$\Pro{\Delta^N(T)\geq g(N)\given A_N(T)}\leq \frac{\E{\Delta^N(T)}}{g(N)}=\frac{A_N(T)}{g(N)}\left(1-\frac{n(N)+1}{N}\right)^{\sss d(N)}.$$
Since $\big\{A^N(T)/N\big\}_{N\geq 1}$ is a tight sequence of random variables, in order to ensure that $\Delta^N(T)/g(N)$ converges to zero in probability, it is enough to have 
\begin{equation}\label{eq:nN-order}
\begin{split}
&\frac{N}{g(N)}\left(1-\frac{n(N)+1}{N}\right)^{\sss d(N)}\to 0\\
\Longleftarrow\hspace{.15cm}&\exp\left(\log\left(\frac{N}{g(N)}\right)-d(N)\frac{n(N)}{N}\right)\to 0\\
\iff& d(N)\frac{n(N)}{N}-\log\left(\frac{N}{g(N)}\right)\to\infty.
\end{split}
\end{equation}
\end{proof}

We now use Propositions~\ref{prop: modified JSL} and~\ref{prop: power of d}  to prove Theorem~\ref{th:pwr of d}.
\begin{proof}[Proof of Theorem~\ref{th:pwr of d}]
Fix any $d(N)$ satisfying either~\eqref{eq:fNalike cond1} or~\eqref{eq:fNalike cond2}.
From Propositions~\ref{prop: modified JSL} and~\ref{prop: power of d} observe that it is enough to 
show that there exists an $n(N)$ with $n(N)\to\infty$ and $n(N)/g(N)\to 0$, as $N\to\infty$, such that
$$\frac{n(N)}{N}d(N)-\log\left(\frac{N}{g(N)}\right)\to\infty.$$

(i) If $g(N) = O(N)$, then observe that $\log(N/g(N))$ is $O(1)$. 
Since $d(N)\to\infty$, choosing $n(N) = N/ \log (d(N))$ satisfies the above criteria, and hence part (i) of the theorem is proved.

(ii) Next we obtain a choice of $n(N)$ if $g(N)=o(N)$.
Note that, if
$$h(N):=\frac{d(N)\frac{g(N)}{N}}{\log\left(\frac{N}{g(N)}\right)}\to\infty,\quad \text{as}\quad N\to\infty,$$
then choosing $n(N)= g(N)/\log(h(N))$, it can be seen that as $N\to\infty$, $n(N)/g(N)\to 0$, and
\begin{equation}\label{eq:equiv-nN-dN}
\begin{split}
&\frac{d(N)\frac{n(N)}{N}}{\log\left(\frac{N}{g(N)}\right)}=\frac{h(N)}{\log(h(N))}\to\infty\\
\implies& \frac{n(N)}{N}d(N)-\log\left(\frac{N}{g(N)}\right)\to\infty.
\end{split}
\end{equation}
\end{proof}

\section{Fluid Limit of JSQ}\label{sec:fluid}
In this section we establish the fluid limit for the ordinary JSQ policy.
In the proof we will leverage the time scale separation technique developed in~\cite{HK94}, suitably extended to an infinite-dimensional space. 
Specifically, note that the rate at which incoming tasks join a server pool
with $i$ active tasks is determined only by the process $\ZZ^N(\cdot)=(Z_1^N(\cdot),\ldots,Z_B^N(\cdot))$, where $Z_i^N(t)=N-Q_i^N(t)$, $i=1,\ldots,B$, represents the number of server pools with fewer than $i$ tasks at time $t$.
Furthermore, in any time interval $[t,t+\varepsilon]$ of length $\varepsilon>0$, the $\ZZ^N(\cdot)$ process experiences $O(\varepsilon N)$ events (arrivals and departures), 
while the $\qq^N(\cdot)$ process can change by only $O(\varepsilon)$ amount.
Therefore, the $\ZZ^N(\cdot)$ process evolves on a much faster time scale than
the $\qq^N(\cdot)$ process.
As a result, in the limit as $N\to\infty$, at each time point $t$, the $\ZZ^N(\cdot)$ process
achieves stationarity depending on the instantaneous value of the $\qq^N(\cdot)$ process, i.e., a separation of time scales takes place. 

In order to illuminate the generic nature of the proof construct,
we will allow for a more general task assignment probability and departure dynamics than described in Section~\ref{sec: model descr}.
Denote by $\bZ_+$ the one-point compactification of the set of nonnegative integers~$\Z_+$, i.e., $\bZ_+=\bZ_+\cup\{\infty\}$.
Equip $\bZ_+$ with the order topology. Denote $G=\bZ_+^B$ equipped with product-topology, and with the Borel $\sigma$-algebra, $\mathcal{G}$.
Let us consider the $G$-valued process $\ZZ^N(s):=(Z_i^N(s))_{i\geq 1}$ as introduced above.
Let $\big\{\mathcal{R}_i\big\}_{1\leq i\leq B}$ be a partition of $G$ such that $\mathcal{R}_i\in\mathcal{G}$.
We assume that a task arriving at (say) $t_k$ is assigned to 
some server pool with $i$ active tasks is given by $p_{i-1}^N(\QQ^N(t_k-))=\ind{\ZZ^N(t_k-)\in\mathcal{R}_{i}}f_i(\qq^N(t_k-))$, where $\ff=(f_1,\ldots,f_B):[0,1]^B\to [0,1]^B$ is Lipschitz continuous, i.e., there exists $C_\ff$, such that for any $\qq_1,\qq_2\in S,$
$$\norm{\ff(\qq_1)-\ff(\qq_2)}\leq C_\ff \norm{\qq_1-\qq_2}.$$
The partition corresponding to the ordinary JSQ policy can be written as
\begin{equation}\label{eq:partition}
\mathcal{R}_i := \big\{(z_1,z_2,\ldots, z_B): z_1=\ldots=z_{i-1}=0<z_i\leq z_{i+1}\leq\ldots\leq z_B\big\},
\end{equation}
with the convention that $Q^N_B$ is always taken to be zero, if $B<\infty$, and $f_i\equiv 1$ for all $i=1,2,\ldots,B$.
The fluid-limit results up to Proposition~\ref{prop:rel compactness} (the relative compactness of the fluid-scaled process) hold true for these general assignment probabilities. 
It is only when proving Theorem~\ref{th:genfluid}, that we need to assume the specific $\big\{\mathcal{R}_i\big\}_{1\leq i\leq B}$ in~\eqref{eq:partition}.
For the departure dynamics, when the system occupancy state is $\QQ^N=(Q_1^N,Q_2^N,\ldots,Q_B^N)$, 
define
the total rate at which departures occur from a server pool with $i$ active tasks by $\mu_i^N(\QQ)$, where $\mmu^N(\QQ)=(\mu_1^N(\QQ),\ldots,\mu_B^N(\QQ))$ will be referred to as the departure rate function.
The departure dynamics described in Section~\ref{sec: model descr} correspond to $\mu_i^N(\QQ)=i(Q_i-Q_{i+1})$ and will be referred to as the infinite-server scenario,
since all active tasks are executed concurrently.
The single-server scenario, where tasks are executed sequentially, corresponds to  the case $\mu_i^N(\QQ)=Q_i-Q_{i+1}$.
\begin{assumption}[{Condition on departure rate function}] \label{assump:mu}
The departure rate function $\mmu^N:\tilde{S}\to[0,\infty)^B$ satisfies the following conditions
\begin{enumerate}[{\normalfont (a)}]
\item There exists a function $\mmu:S\to[0,\infty)^B$, such that 
$$\lim_{N\to\infty}\sup_{\qq\in S^N}\norm{\frac{1}{N}\mmu^N(\lfloor N\qq\rfloor)-\mmu(\qq)}=0.$$
\item The function $\mmu$ is Lipschitz continuous in $S$, i.e., there exists a constant $C_{\mmu}$, such that for any $\qq_1,\qq_2\in S$, 
$$\norm{\mmu(\qq_1)-\mmu(\qq_2)}\leq C_{\mmu} \norm{\qq_1-\qq_2}.$$
\item Also, $\mmu^N$ satisfies linear growth constraints in each coordinate, i.e., for all $i\geq 1$, there exists $C_i>0$, such that for all $\qq\in S$,
 $$\mu^N_i(\lfloor N\qq\rfloor)\leq NC_i(1+\norm{\qq}).$$
We will often omit $\lfloor \cdot\rfloor$ in the argument of $\mmu^N$ for notational convenience.
\end{enumerate}
\end{assumption}
Under these assumptions on the departure rate function, we prove the following fluid-limit result for the ordinary JSQ policy.
Recall the definition of $m(\qq)$ in Subsection~\ref{ssec:fluid}, and define
\begin{equation}\label{eq:fluid-gen}
p_{i}(\qq)=
\begin{cases}
\min\big\{\mu_{m(\qq)}(\qq)/\lambda,1\big\} & \quad\mbox{ for }\quad i=m(\qq)-1,\\
1 - p_{\sss m(\qq) - 1}(\qq) & \quad\mbox{ for }\quad i=m(\qq),\\
0&\quad \mbox{ otherwise.}
\end{cases}
\end{equation}
Note that $p_i(\cdot)$ in~\eqref{eq:fluid-gen} is consistent with the one defined in Subsection~\ref{ssec:fluid} for the proper choice of the departure rate function $\mu_i(\qq)=i(q_i-q_{i+1})$.
\begin{theorem}[{Fluid limit of JSQ}]
\label{th:genfluid}
Assume $\qq^N(0)\pto \qq^\infty\in S$ and $\lambda(N)/N\to\lambda>0$ as $N\to\infty$. Further assume that the departure rate function $\mmu^N$ satisfies Assumption~\ref{assump:mu}. Then the sequence of processes $\big\{\qq^N(t)\big\}_{t\geq 0}$ for the ordinary JSQ policy has a continuous weak limit that satisfies the system of integral equations
\begin{equation}\label{eq:fluidfinal}
q_i(t) = q_i(0)+\lambda \int_0^t p_{i-1}(\qq(s))\dif s - \int_0^t \mu_i(\qq(s))\dif s, \quad i=1,2,\ldots,B,
\end{equation}
where $\qq(0) = \qq^\infty$ and the coefficients $p_i(\cdot)$ are
defined in~\eqref{eq:fluid-gen}, and may be interpreted as the fractions of incoming tasks
assigned to server pools with exactly $i$ active tasks.
\end{theorem}
We will now verify that the departure rate functions corresponding to the infinite-server and single-server scenarios satisfy the conditions in Assumption~\ref{assump:mu}.
\begin{proposition}
The following departure rate functions denoted by $\mmu = (\mu_1, \mu_2,\ldots,\mu_B)$, satisfy the conditions in Assumption~\ref{assump:mu}.
For $\QQ\in \tilde{S}$, and $\qq\in S$,
\begin{enumerate}[{\normalfont (i)}]
\item $\mu_i^N(\QQ)=Q_i-Q_{i+1}$, and $\mu_i(\qq)=q_i-q_{i+1}$, $i\geq 1$.
\item $\mu_i^{N}(\QQ)=i(Q_i-Q_{i+1})$, and $\mu_i(\qq)=i(q_i-q_{i+1})$, $i\geq 1$.
\end{enumerate}
\end{proposition}
\begin{proof}
Observe that if $B<\infty$, then since componentwise $\mu_i$ satisfies all the conditions for all $i\geq 1$, $\mmu$ satisfies the conditions in the product space as well.
Therefore, let us consider the case when $B=\infty$.
In this case observe that,
for both (i) and (ii) condition (a) is immediate, since $\mmu^N(\lfloor N\qq\rfloor)/N=\mmu(\qq)$ for all $\qq\in S^N.$
Also, the linear growth rate constraint in condition~(c) is satisfied in both cases by taking $C_i=1$ in~(i) and $C_i=i$ in~(ii).

Now we will show that in both cases $\mmu$ is Lipschitz continuous in $S$.

(i) For $\mu_i(\qq)=q_i-q_{i+1}$, $i\geq 1$, and $\qq_1,\qq_2\in S$,
$$
\norm{\mmu(\qq)}=\sum_{i\geq 1}\frac{|q_{i}-q_{i+1}|}{2^i}
\leq \sum_{i\geq 1}\frac{q_{i}}{2^i}+\sum_{i\geq 1}\frac{q_{i+1}}{2^i}
\leq 2\norm{\qq}.
$$

(ii) Now assume $\mu_i(\qq)=i(q_i-q_{i+1})$, $i\geq 1$. 
Since $\mmu$ is a linear operator on the Banach space (complete normed linear space) $\R^B$, to prove Lipschitz continuity of $\mu$, it is enough to show that $\mu$ is continuous at zero. 
Specifically, we will show that for any sequence $\big\{\qq^n\big\}_{n\geq 1}$, in $\R^B$, $\norm{\qq^n}\to 0$ implies $\norm{\mmu(\qq^n)}\to 0$. 
This would imply that there exists $\varepsilon>0$, such that whenever $\norm{\qq^n}\leq\varepsilon$ with $\qq^n\in \R^B$, we have $\norm{\mmu(\qq^n)}<1$. Then due to linearity of $\mmu$, for any $\qq\in \R^B$,
\begin{align*}
\norm{\mmu(\qq)}&=\norm{\frac{\norm{\qq}}{\varepsilon}\mmu\left(\varepsilon\frac{\qq}{\norm{\qq}}\right)}\\
&\leq \frac{\norm{\qq}}{\varepsilon}\norm{\mmu\left(\varepsilon\frac{\qq}{\norm{\qq}}\right)}\\
&\leq \frac{1}{\varepsilon}\norm{\qq}.
\end{align*}
To show that $\mmu$ is continuous at $\mathbf{0}\in\R^B$, fix
any $\varepsilon>0$. 
Also, fix an $M>0$, depending upon~$\varepsilon$, such that $\sum_{i> M}1/2^i<\varepsilon/2$. Now, choose $\delta<\varepsilon/(4M)$. Then, for any $\qq$ such that $\norm{\qq}<\delta$, we have
\begin{align*}
\norm{\mmu(\qq)}&=\sum_{i=1}^\infty \frac{i|q_i-q_{i+1}|}{2^i}\\
&=\sum_{i=1}^M \frac{i|q_i-q_{i+1}|}{2^i}+\frac{\varepsilon}{2}\\
&\leq M\sum_{i=1}^M \frac{|q_i-q_{i+1}|}{2^i}+\frac{\varepsilon}{2}\\
&\leq 2M\norm{\qq}+\frac{\varepsilon}{2}\leq \varepsilon.
\end{align*}
Hence, $\mmu$ is Lipschitz continuous on $\R^\infty$.
\end{proof}

\subsection{Martingale representation}
In this subsection we construct the martingale representation of the occupancy state process $\QQ^N(\cdot)$.
The component $Q_i^N(t)$, satisfies the identity relation
\begin{align}\label{eq:recursion}
Q_i^N(t)=Q_i^N(0)+A_i^N(t)-D_i^N(t),&\quad\mbox{ for }\quad i=1,\ldots, B,
\end{align}
where
\begin{align*}
A_i^N(t)&=\mbox{ number of arrivals during $[0,t]$ to some server pool with }i-1\mbox{ active tasks,} \\
D_i^N(t)&=\mbox{ number of departures during $[0,t]$ from some server pool with }i\mbox{ active tasks}.
\end{align*}
We can express $A^N_i(t)$ and $D_i^N(t)$ as
\begin{align*}
A^N_i(t) &=  \mathcal{N}_{A,i}\left(\lambda(N)\int_0^t p_{i-1}^N(\QQ^N(s))\dif s\right),\\
D_i^N(t) &=  \mathcal{N}_{D,i}\left(\int_0^t \mu_i^N(\QQ^N(s))\dif s\right),
\end{align*}
where $\mathcal{N}_{A,i}$ and $\mathcal{N}_{D,i}$ are mutually independent unit-rate Poisson processes, $i=1,2,\ldots,B$.
Define the following sigma fields.
\begin{align*}
\mathcal{A}^N_i(t)&:= \sigma\left(A^N_i(s): 0\leq s\leq t\right),\\
\mathcal{D}_i^N(t)&:= \sigma\left(D_i^N(s): 0\leq s\leq t\right),\mbox{ for }i\geq 1,
\end{align*}
and the filtration $\mathbf{F}^N\equiv\big\{\mathcal{F}^N_t:t\geq 0\big\}$ with
\begin{equation}\label{eq:filtration}
\mathcal{F}^N_t:=\bigvee_{i=1}^{\infty} \big[\mathcal{A}_i^N(t)\vee \mathcal{D}_i^N(t)\big]
\end{equation}
augmented by all the null sets. 
Now we have the following martingale decomposition from the classical result in \cite[Proposition~3]{D76}.

\begin{proposition}\label{prop:mart-rep}
The following are $\mathbf{F}^N$-martingales, for $i\geq 1$:
\begin{equation}\label{eq:martingales}
\begin{split}
M_{A,i}^N(t)&:=  \mathcal{N}_{A,i}\left(\lambda(N)\int_0^t p_{i-1}^N(\QQ^N(s))\dif s\right)-\lambda(N)\int_0^t p_{i-1}^N(\mathbf{Q}^N(s)) \dif s,\\
M_{D,i}^N(t)&:=\mathcal{N}_{D,i}\left(\int_0^t \mu_i^N(\QQ^N(s))\dif s\right)-\int_0^t \mu_i^N(\QQ^N(s))\dif s,
\end{split}
\end{equation}
with respective compensator and predictable quadratic variation processes given by
\begin{align*}
\langle M_{A,i}^N\rangle(t)&:= \lambda(N)\int_0^t p_{i-1}^N(\mathbf{Q}^N(s-))\dif s,\\
\langle M_{D,i}^N\rangle(t)&:=\int_0^t \mu_i^N(\QQ^N(s))\dif s.
\end{align*}
\end{proposition}

Therefore, finally we have the following martingale representation of the $N^{\mathrm{th}}$ process:
\begin{equation}\label{eq:mart-unscaled}
\begin{split}
Q_i^N(t)&=Q_i^N(0)+\lambda(N)\int_0^t p_{i-1}^N(\mathbf{Q}^N(s))\dif s\\
&-\int_0^t \mu_i^N(\QQ^N(s))\dif s +(M_{A,i}^N(t)-M_{D,i}^N(t)),\quad t\geq 0,\quad i= 1,\ldots,B.
\end{split}
\end{equation}
In the proposition below, we prove that the martingale part vanishes when scaled by $N$. Since convergence in probability in each component implies convergence in probability with respect to the product topology, it is enough to show convergence in each component.

\begin{proposition}\label{prop:mart zero1}
For all $i\geq 1$,
$$\left\{\frac{1}{N}(M_{A,i}^N(t)-M_{D,i}^N(t))\right\}_{t\geq 0}\dto \big\{m(t)\big\}_{t\geq 0}\equiv 0.$$
\end{proposition}
\begin{proof}
Fix any $T\geq 0$, and $i\geq 1$. From Doob's inequality \cite[Theorem 1.9.1.3]{LS89}, we have
\begin{align*}
\Pro{\sup_{t\in[0,T]}\frac{1}{N}M_{A,i}^N(t)\geq \epsilon}&=\Pro{\sup_{t\in[0,T]}M_{A,i}^N(t)\geq N\epsilon}\\
&\leq \frac{1}{N^2\epsilon^2}\E{\langle M_{A,i}^N\rangle (T)}\\
&\leq \frac{1}{N\epsilon^2}\int_0^T p_{i-1}(\mathbf{Q}^N(s-))\lambda N\dif s\\
&\leq \frac{\lambda T}{N\epsilon^2}\to 0,\mbox{ as }N\to\infty.
\end{align*}
Similarly, for $M_{D,i}^N$, 
\begin{align*}
\Pro{\sup_{t\in[0,T]}\frac{1}{N}M_{D,i}^N(t)\geq \epsilon}&=\Pro{\sup_{t\in[0,T]}M_{D,i}^N(t)\geq N\epsilon}\\
&\leq \frac{1}{N^2\epsilon^2}\E{\langle M_{D,i}^N\rangle (T)}\\
&\leq \frac{1}{N^2\epsilon^2}\int_0^T \mu_i^N(\QQ^N(s))\dif s\\
&\leq \frac{2L'}{N\epsilon^2}\to 0,\mbox{ as }N\to\infty,
\end{align*}
where the last inequality follows from the linear growth constraint in Assumption~\ref{assump:mu}~(c).
Therefore we have uniform convergence over compact sets, and hence with respect to the Skorohod-$J_1$ topology. 
\end{proof}

\subsection{Relative compactness and uniqueness}\label{subsec:fluidlimit}
Now we will first prove the relative compactness of the sequence of fluid-scaled processes.
Recall that we denote all the fluid-scaled quantities by their respective small letters, e.g.~$\mathbf{q}^N(t):=\mathbf{Q}^N(t)/N$, componentwise, i.e., $q_i^N(t):=Q_i^N(t)/N$ for $i\geq 1$. Therefore the martingale representation in~\eqref{eq:mart-unscaled}, can be written as
\begin{equation}\label{eq:mart1}
\begin{split}
q_i^N(t)&=q_i^N(0)+\frac{\lambda(N)}{N}\int_0^t p_{i-1}^N(\mathbf{Q}^N(s))\dif s\\
&-\int_0^t \frac{1}{N}\mu_i^N(\QQ^N(s))\dif s +\frac{1}{N}(M_{A,i}^N(t)-M_{D,i}^N(t)),\quad i=1,2,\ldots, B,
\end{split}
\end{equation}
or equivalently,
\begin{equation}\label{eq:martingale rep assumption 2}
\begin{split}
q_i^N(t)&=q_i^N(0)+\frac{\lambda(N)}{N}\int_0^t f_i(\qq^N(s))\ind{\ZZ^N(s)\in\mathcal{R}_{i}}\dif s\\
&-\int_0^t \frac{1}{N}\mu_i^N(\QQ^N(s))\dif s +\frac{1}{N}(M_{A,i}^N(t)-M_{D,i}^N(t)),\quad i=1,2,\ldots, B.
\end{split}
\end{equation}
Now, we consider the Markov process $(\qq^N,\ZZ^N)(\cdot)$ defined on $S\times G$. 
Define a random measure $\alpha^N$ on the measurable space $([0,\infty)\times G, \mathcal{C}\otimes\mathcal{G})$, when $[0,\infty)$ is endowed with Borel sigma algebra $\mathcal{C}$, by
\begin{equation}
\alpha^N(A_1\times A_2):=\int_{A_1} \ind{\ZZ^N(s)\in A_2}\dif s,
\end{equation}
for $A_1\in\mathcal{C}$ and $A_2\in\mathcal{G}$. 
Then the representation in \eqref{eq:martingale rep assumption 2} can be written in terms of the random measure as,
\begin{equation}\label{eq:martingale rep assumption 2-2}
\begin{split}
q_i^N(t)&=q_i^N(0)+\lambda\int_{[0,t]\times\mathcal{R}_i} f_i(\qq^N(s))\dif\alpha^N\\
&-\int_0^t \frac{1}{N}\mu_i^N(\QQ^N(s))\dif s +\frac{1}{N}(M_{A,i}^N(t)-M_{D,i}^N(t)),\quad i=1,2,\ldots, B.
\end{split}
\end{equation}
Let $\mathfrak{L}$ denote the space of all measures on $[0,\infty)\times G$ satisfying $\gamma([0,t],G)= t$, endowed with the topology corresponding to weak convergence of measures restricted to $[0,t]\times G$ for each $t$.
\begin{proposition}\label{prop:rel compactness}
Assume $\mathbf{q}^N(0)\dto\mathbf{q}(0)$ as $N\to\infty$, then $\big\{(\mathbf{q}^N(\cdot),\alpha^N)\big\}$ is a relatively compact sequence in $D_{S}[0,\infty)\times\mathfrak{L}$ and the limit $\big\{(\mathbf{q}(\cdot),\alpha)\big\}$ of any convergent subsequence satisfies
\begin{equation}\label{eq:rel compact}
q_i(t)=q_i(0)+\lambda\int_{[0,t]\times\mathcal{R}_i} f_i(\qq(s))\dif\alpha -\int_0^t \mu_i(\qq(s))\dif s,\quad i=1,2,\ldots, B.
\end{equation}
\end{proposition}
\begin{remark}{\normalfont
Proposition~\ref{prop:rel compactness} is true even when the function $\ff$ in the assignment probability depends on $N$. In that case the proof will go through by assuming that $\ff^N$ converges uniformly to some Lipschitz continuous function $\ff$ in the sense of Assumption~\ref{assump:mu}.(a).
}
\end{remark}
\begin{remark}{\normalfont
The relative compactness result in the above proposition holds for an even more general class of assignment probabilities than those considered above.
Since the proof will follow a nearly identical line of arguments, we briefly mention them here. 
Consider a scheme for which the assignment probabilities can be written as
$$p_{i}^N(\QQ^N) = \eta_1 \ind{\ZZ^N\in\mathcal{R}_i}+\eta_2 g_i(\qq^N),\quad i=1,\ldots,B,$$
for some fixed $\eta_1,\eta_2\in[0,1]$, and some Lipschitz continuous function $\gb =(g_1,g_2,\ldots,g_B):S\to[0,\infty)^B$.
The above scheme assigns a fixed fraction $\eta_1$ of incoming tasks according to the ordinary JSQ policy, and a fraction $\eta_2$ as some suitable function of the fluid-scaled occupancy states $\gb(\qq)$, for $\qq\in S$.
In practice, the above scheme can handle (two or more) priorities among the incoming tasks, by assigning the high-priority tasks in accordance with the ordinary JSQ policy, and others governed by the JSQ$(d)$ scheme, say.
In that case, the fluid limit in~\eqref{eq:rel compact} will become
\begin{equation}\label{eq:rel compact-gen}
q_i(t)=q_i(0)+\lambda\eta_1\alpha([0,t]\times\mathcal{R}_i)+\eta_2\int_0^t g_i(\qq(s))\dif s -\int_0^t \mu_i(\qq(s))\dif s,\quad i=1,2,\ldots, B.
\end{equation}
}
\end{remark}
To prove Proposition~\ref{prop:rel compactness}, we will verify the conditions of relative compactness from~\cite{EK2009}. 
Let $(E,r)$ be a complete and separable metric space. For any $x\in D_E[0,\infty)$, $\delta>0$ and $T>0$, define
\begin{equation}\label{eq:mod-continuity}
w'(x,\delta,T)=\inf_{\{t_i\}}\max_i\sup_{s,t\in[t_{i-1},t_i)}r(x(s),x(t)),
\end{equation}
where $\big\{t_i\big\}$ ranges over all partitions of the form $0=t_0<t_1<\ldots<t_{n-1}<T\leq t_n$ with $\min_{1\leq i\leq n}(t_i-t_{i-1})>\delta$ and $n\geq 1$.
 Below we state the conditions for the sake of completeness.
\begin{theorem}\label{th:from EK}
\begin{normalfont}
\cite[Corollary~3.7.4]{EK2009}
\end{normalfont}
Let $(E,r)$ be complete and separable, and let $\big\{X_n\big\}_{n\geq 1}$ be a family of processes with sample paths in $D_E[0,\infty)$. Then $\big\{X_n\big\}_{n\geq 1}$ is relatively compact if and only if the following two conditions hold:
\begin{enumerate}[{\normalfont (a)}]
\item For every $\eta>0$ and rational $t\geq 0$, there exists a compact set $\Gamma_{\eta, t}\subset E$ such that $$\varliminf_{n\to\infty}\Pro{X_n(t)\in\Gamma_{\eta, t}}\geq 1-\eta.$$
\item For every $\eta>0$ and $T>0$, there exists $\delta>0$ such that
$$\varlimsup_{n\to\infty}\Pro{w'(X_n,\delta, T)\geq\eta}\leq\eta.$$
\end{enumerate}
\end{theorem}

\begin{proof}[Proof of Proposition~\ref{prop:rel compactness}]
The proof goes in two steps. We first prove the relative compactness, and then show that the limit satisfies~\eqref{eq:rel compact}.

Observe from \cite[Proposition 3.2.4]{EK2009} that, to prove the relative compactness of the process $\big\{(\mathbf{q}^N(\cdot),\alpha^N)\big\}$, it is enough to prove relative compactness of the individual components.
Note that, from Prohorov's theorem~\cite[Theorem 3.2.2]{EK2009}, $\mathfrak{L}$ is compact, since $G$ is compact. Now, relative compactness of $\alpha^N$ follows from the compactness of $\mathfrak{L}$ under the topology of weak convergence of measures and Prohorov's theorem.

To claim the relative compactness of $\big\{\mathbf{q}^N(\cdot)\big\}$, first observe that $[0,1]^{B}$ is compact with respect to product topology, and $S$ is a closed subset of $[0,1]^B$, and hence $S$ is also compact with respect to product topology. So, the compact containment condition (a) of Theorem~\ref{th:from EK} is satisfied by taking $\Gamma_{\eta,t}\equiv S$.

For condition (b), we will show for each coordinate $i$, that for any $\eta>0$, there exists $\delta>0$, such that for any $t_1,t_2>0$ with $|t_1-t_2|<\delta$,
\begin{equation*}
\varlimsup_{n\to\infty}\Pro{|q^n_i(t_1)-q^n_i(t_2)|\geq \eta}=0.
\end{equation*}
With respect to product topology, this will imply that for any $\eta>0$, there exists $\delta>0$, such that for any $t_1,t_2>0$ with $|t_1-t_2|<\delta$,
\begin{equation*}
\varlimsup_{n\to\infty}\Pro{\norm{q^n(t_1)-q^n(t_2)}\geq \eta}=0,
\end{equation*}
which in turn will imply condition~(b) in Theorem~\ref{th:from EK}.
To see this, observe that for any fixed $\eta>0$ and $T>0$, we can
choose $\delta'>0$ small enough, so that for any fine enough finite partition $0=t_0<t_1<\ldots<t_{n-1}<T\leq t_n$ of $[0,T]$ with $\min_{1\leq i\leq n}(t_i-t_{i-1})>\delta'$ and $\max_{1\leq i\leq n}(t_i-t_{i-1})<\delta$,  $\varlimsup_{n\to\infty}\Pro{\norm{q^n(t_i)-q^n(t_{i+1})}\geq \eta}=0$ for all $1\leq i\leq n$.

Now fix any $0\leq t_1<t_2<\infty$, and  $1\leq i\leq B$.
\begin{align*}
&|q_i^N(t_1)-q_i^N(t_2)|\\
&\leq \lambda \alpha^N([t_1,t_2]\times\mathcal{R}_i)+\int_{t_1}^{t_2} \frac{1}{N}\mu_i^N(\QQ^N(s))\dif s \\
&+\frac{1}{N}|M_{A,i}^N(t_1)-M_{D,i}^N(t_1)-M_{A,i}^N(t_2)+M_{D,i}^N(t_2)|\\
&\leq\lambda'(t_2-t_1)+\frac{1}{N}|M_{A,i}^N(t_1)-M_{D,i}^N(t_1)-M_{A,i}^N(t_2)+M_{D,i}^N(t_2)|,
\end{align*}
for some $\lambda'\in\R$, using the linear growth constraint of $\mmu^N$ due to Assumption~\ref{assump:mu}(c). 
Now, from Proposition~\ref{prop:mart zero1}, we get, for any $T\geq 0$,
$$\sup_{t\in[0,T]}\frac{1}{N}|M_{A,i}^N(t_1)-M_{D,i}^N(t_1)-M_{A,i}^N(t_2)+M_{D,i}^N(t_2)|\pto 0.$$
To prove that the limit $\big\{(\qq(\cdot),\alpha)\big\}$ of any convergent subsequence satisfies~\eqref{eq:rel compact}, we will use the continuous-mapping theorem~\cite[Section~3.4]{W02}.
Specifically, we will show that the right side of~\eqref{eq:martingale rep assumption 2-2} is a continuous map of suitable arguments.
Let $\big\{\qq(t)\big\}_{t\geq 0}$ and $\big\{\yy(t)\big\}_{t\geq 0}$ be an $S$-valued and an $\R^B$-valued c\'adl\'ag function, respectively. 
Also, let $\alpha$ be a measure on the measurable space $([0,\infty)\times G, \mathcal{C}\otimes\mathcal{G})$. Then for $\qq^0\in S$, define for $i\geq 1$,
$$F_i(\qq,\alpha,\qq^0,\yy)(t):=q_i^0+y_i(t)+\lambda\int_{[0,t]\times\mathcal{R}_i} f_i(\qq(s))\dif\alpha-\int_0^t\mu_i(\qq(s))\dif s.$$
Observe that it is enough to show $\FF=(F_1,\ldots,F_B)$ is a continuous
operator. Indeed, in that case the right side of~\eqref{eq:martingale rep assumption 2-2} can be written as $\FF(\qq^N,\alpha^N,\qq^N(0),\yy^N)$, where $\yy^N=(y_1^N,\ldots,y_B^N)$ with $y_i^N= (M_{A,i}^N-M_{D,i}^N)/N$, and since each argument converges we will get the convergence to the right side of~\eqref{eq:rel compact}.
Therefore, we now prove the continuity of $\FF$ below. 
In particular assume that the sequence of processes $\big\{(\qq^N,\yy^N)\big\}_{N\geq 1}$ converges to $\big\{(\qq,\yy)\big\}$, for any fixed $t\geq 0$, the measure $\alpha^N([0,t],\cdot)$ on $G$ converges weakly to $\alpha([0,t],\cdot)$, and the sequence of $S$-valued random  variables $\qq^N(0)$
 converges weakly to $\qq(0)$.
 Fix any $T\geq 0$ and $\varepsilon>0$.
 \begin{enumerate}[{\normalfont (i)}]
 \item  Choose $N_1\in\N$, such that $\sup_{t\in[0,T]}\norm{\qq^N(t)-\qq(t)}<\varepsilon/(4TC_{\mmu})$. In that case, observe that
 \begin{align*}
 \sup_{t\in [0,T]}\int_0^t\norm{\mmu(\qq^N(s))-\mmu(\qq(s))}\dif s& \leq T\sup_{t\in [0,T]}\norm{\mmu(\qq^N(t))-\mmu(\qq(t))}\\
 & \leq TC_{\mmu}\sup_{t\in [0,T]}\norm{\qq^N(t))-\qq(t)}<\frac{\varepsilon}{4},
 \end{align*}
 where we have used the Lipschitz continuity of $\mmu$ due to Assumption~\ref{assump:mu}(b).
 \item Choose $N_2\in\N$, such that $\sup_{t\in[0,T]}\norm{\yy^N(t)-\yy(t)}<\varepsilon/4$,
 \item   Choose $N_3\in\N$, such that 
 $$\sum_{i\geq 1}\frac{\lambda}{2^i}\left|\int_{[0,T]\times\mathcal{R}_i} f_i(\qq^N(s))\dif\alpha^N-\int_{[0,T]\times\mathcal{R}_i} f_i(\qq(s))\dif\alpha\right|<\frac{\varepsilon}{4}.$$
 This can be done as follows: choose $M\in \N$ large enough so that $\sum_{i> M}2^{-i}<\varepsilon/8.$ Now for $i\leq M$, since $\alpha^N([0,T],\cdot)$ converges weakly to $\alpha([0,T],\cdot)$, and $M$ is finite, we can choose $N_3\in\N$ such that 
 \begin{align*}
 &\sum_{i= 1}^M\frac{\lambda}{2^i}\left|\int_{[0,T]\times\mathcal{R}_i} f_i(\qq^N(s))\dif\alpha^N-\int_{[0,T]\times\mathcal{R}_i} f_i(\qq(s))\dif\alpha\right|\\
 \leq & \sum_{i= 1}^M\frac{\lambda}{2^i}\int_{[0,T]\times\mathcal{R}_i} |f_i(\qq^N(s))-f_i(\qq(s))|\dif\alpha^N
 +\sum_{i = 1}^M\frac{\lambda}{2^i}|\alpha^N([0,T]\times \mathcal{R}_i)-\alpha([0,T]\times \mathcal{R}_i)|\\
 \leq &  \sum_{i= 1}^M\frac{\lambda }{2^i}TC_\ff\sup_{s\in [0,T]} \norm{\qq^N(s)-\qq(s)}
 +\sum_{i = 1}^M\frac{\lambda}{2^i}|\alpha^N([0,T]\times \mathcal{R}_i)-\alpha([0,T]\times \mathcal{R}_i)|<\frac{\varepsilon}{4}.
 \end{align*}
 \item  Choose $N_4\in\N$, such that $\norm{\qq^N(0)-\qq(0)}<\varepsilon/4$.
 \end{enumerate}
Let $\hat{N}=\max\big\{N_1,N_2,N_3,N_4\big\}$, then for $N\geq \hat{N}$,
\begin{align*}
&\sup_{t\in [0,T]}\norm{\FF(\qq^N,\alpha^N,\qq^N(0),\yy^N)-\FF(\qq,\alpha,\qq(0),\yy)}(t)<\varepsilon.
\end{align*}
Thus the proof of continuity of $\FF$ is complete.
\end{proof}

To characterize the limit in~\eqref{eq:rel compact}, for any $\qq\in S$, define the Markov process  $\ZZ_{\qq}$ on $G$ as
\begin{equation}\label{eq:slowprocess}
\ZZ_{\qq} \rightarrow 
\begin{cases}
\ZZ_{\qq}+e_i& \quad\mbox{ at rate }\quad \mu_i(\qq)\\
\ZZ_{\qq}-e_i& \quad\mbox{ at rate }\quad \lambda \ind{\ZZ_\qq\in\mathcal{R}_{i}},
\end{cases}
\end{equation}
where $e_i$ is the $i^{\mathrm{th}}$ unit vector, $i=1,\ldots,B$.

\begin{proof}[{Proof of Theorem~\ref{th:genfluid}}]
Having proved the relative compactness in Proposition~\ref{prop:rel compactness},  it follows from analogous arguments as used in the proof of~\cite[Theorem 3]{HK94}, that the limit of any convergent subsequence of the sequence of processes $\big\{\qq^N(t)\big\}_{t\geq 0}$ satisfies
\begin{equation}
q_i(t) = q_i(0)+\lambda \int_0^t \pi_{\qq(s)}(\mathcal{R}_i)\dif s - \int_0^t \mu_i(\qq(s))\dif s, \quad i=1,2,\ldots,B,
\end{equation}
for \emph{some} stationary measure $\pi_{\qq(t)}$ of the Markov process  $\ZZ_{\qq(t)}$ described in~\eqref{eq:slowprocess} satisfying $\pi_{\qq}\big\{\ZZ: Z_i=\infty\big\}=1$ if $q_i<1$. 

Now it remains to show that $\qq(t)$ \emph{uniquely} determines $\pi_{\qq(t)}$, and that $\pi_{\qq(s)}(\mathcal{R}_i)=p_{i-1}(\qq(s))$ described in~\eqref{eq:fluid-gen}. 
As mentioned earlier, in this proof we will now assume the specific assignment probabilities in~\eqref{eq:partition}, corresponding to the ordinary JSQ policy.
To see this, fix any $\qq=(q_1,\ldots,q_B)\in S$, 
and assume that there exists $m\geq 0$, such that $q_{m+1}<1$ and $q_1=\ldots=q_m=1$,
with the convention that $q_0\equiv 1$ and $q_{B+1}\equiv 0$ if $B<\infty$. In that case,
$$\pi_{\qq}\big(\big\{Z_{m+1}=\infty, Z_{m+2}=\infty,\ldots,Z_B=\infty\big\}\big)=1.$$
Also, 
note that $q_i = 1$ forces $\dif q_i/\dif t \leq 0$, i.e., $\lambda \pi_{\qq}(\mathcal{R}_i) \leq \mu_i(\qq)$ for all $i = 1, \ldots, m$, and in particular $\pi_{\qq}(\mathcal{R}_i) = 0$ for all $i = 1, \ldots, m - 1.$ Thus,
$$\pi_{\qq}\big(\big\{Z_1=0,Z_2=0,\ldots,Z_{m-1}=0\big\}\big)=1.$$

Therefore, $\pi_\qq$ is determined only by the stationary distribution of the $m^{\mathrm{th}}$ component, which can be described as a birth-death process
\begin{equation}\label{eq:bdprocess}
Z \rightarrow 
\begin{cases}
Z+1& \quad\mbox{ at rate }\quad \mu_m(\qq)\\
Z-1& \quad\mbox{ at rate }\quad \lambda\ind{Z>0}
\end{cases}
\end{equation}
and let $\pi^{(m)}$ be its stationary distribution. 
Now it is enough to show that $\pi^{(m)}$ is uniquely determined by $\mu_m(\qq)$. 
First observe that the process on $\bZ$ described in~\eqref{eq:bdprocess} is reducible, and can be decomposed into
two irreducible classes given by $\mathbbm{Z}$ and $\{\infty\}$, respectively.
Therefore, if $\pi^{(m)}(Z=\infty)=0$ or $1$, then it is unique. 
Indeed, if $\pi^{(m)}(Z=\infty)=0$, then $Z$ is birth-death process on $\mathbbm{Z}$ only, and hence it has a unique stationary distribution. 
Otherwise, if $\pi^{(m)}(Z=\infty)=1$, then it is trivially unique. 
Now we distinguish between two cases depending upon whether $\mu_m(\qq)\geq \lambda$ or not.

Note that if $\mu_m(\qq)\geq\lambda$, then $\pi^{(m)}(Z\geq k)=1$ for all $k\geq 0$. 
On $\bZ$ this shows that $\pi^{(m)}(Z=\infty)=1$.
Furthermore, if $\mu_m(\qq)<\lambda$, we will show that $\pi^{(m)}(Z=\infty)=0$.
On the contrary, assume $\pi^{(m)}(Z=\infty)=\varepsilon\in (0,1]$.
Also, let $\hat{\pi}^{(m)}$ be the unique stationary distribution of the birth-death process in~\eqref{eq:bdprocess} restricted to $\mathbbm{Z}$.
Therefore, 
$$\pi^{(m)}(Z>0)=\hat{\pi}^{(m)}(Z>0)+\varepsilon>\hat{\pi}^{(m)}(Z>0)=\frac{\mu_m(\qq)}{\lambda},$$
and $\pi_\qq(\mathcal{R}_m)=\pi^{(m)}(Z>0)>\mu_m(\qq)/\lambda$.
Putting this value in the fluid-limit equation~\eqref{eq:fluidfinal}, we obtain that $\dif q_m(t)/\dif t>0$.
Since $q_m(t)=1$, this leads to a contradiction, and hence it must be the case that $\pi^{(m)}(Z=\infty)=0$. 

Therefore, for all $\qq\in S$, $\pi_\qq$ is uniquely determined by $\qq$. 
Furthermore, we can identify the expression for $\pi_q(\mathcal{R}_i)$ as
\begin{equation}
\pi_\qq(\mathcal{R}_i)=
\begin{cases}
\min\big\{\mu_i(\qq)/\lambda,1\big\}& \quad\mbox{ for }\quad i=m,\\
1- \min\big\{\mu_i(\qq)/\lambda,1\big\} & \quad\mbox{ for }\quad i=m+1,\\
0&\quad \mbox{ otherwise,}
\end{cases}
\end{equation}
and hence $\pi_{\qq(s)}(\mathcal{R}_i)=p_{i-1}(\qq(s))$ as claimed.
\end{proof}

\section{Diffusion Limit of JSQ: Non-integral \texorpdfstring{$\boldsymbol{\lambda}$}{lambda}}\label{sec:non-itegral}
In this section we establish the diffusion-scale behavior of the ordinary JSQ policy in the case when $\lambda$ is not an integer, i.e., $f>0$.
Recall that $f(N)=\lambda(N)-KN.$
In this regime, let us define the following centered and scaled processes:
\begin{equation}
\begin{split}
\bar{Q}^N_i(t)&=N-Q^N_i(t)\geq 0\quad \mathrm{for}\quad i\leq K-1\\
\bar{Q}_{K}^N(t)&:=\frac{N-Q_K^N(t)}{\log (N)}\geq 0\\
\bar{Q}_{K+1}^N(t)&:=\frac{Q^N_{K+1}(t)-f(N)}{\sqrt{N}}\in\R\\
\bar{Q}^N_i(t)&:=Q^N_i(t)\geq 0\quad \mathrm{for}\quad i\geq K+2.
\end{split}
\end{equation}
\begin{theorem}\label{th:diffusion}{\normalfont [Diffusion limit for JSQ policy; $f>0$]}
Assume $\bar{Q}^N_i(0)\dto \bar{Q}_i(0)$ in $\R$, $i\geq 1$, and $\lambda(N)/N\to\lambda>0$ as $N\to\infty$, with $f=\lambda-\lfloor\lambda\rfloor>0$, then 
\begin{enumerate}[{\normalfont(i)}]
\item $\lim_{N\to\infty}\Pro{\sup_{t\in[0,T]}\bar{Q}_{K-1}^N(t)\leq 1}=1$, and $\big\{\bar{Q}^N_i(t)\big\}_{t\geq 0}\dto \big\{\bar{Q}_i(t)\big\}_{t\geq 0}$, where $\bar{Q}_i(t)\equiv 0$, provided 
$\lim_{N\to\infty}\Pro{\bar{Q}_{K-1}^N(0)\leq 1}=1$, and
$\bar{Q}_i^N(0)\pto 0$ for $i\leq K-2$.
\item $\big\{\bar{Q}^N_K(t)\big\}_{t\geq 0}$ is a stochastically bounded sequence of processes in $D_{\R}[0,\infty)$.
\item $\big\{\bar{Q}^N_{K+1}(t)\big\}_{t\geq 0}\dto \big\{\bar{Q}_{K+1}(t)\big\}_{t\geq 0}$, where $\bar{Q}_{K+1}(t)$ is given by the Ornstein-Uhlenbeck process satisfying the following stochastic differential equation:
$$d\bar{Q}_{K+1}(t)=-\bar{Q}_{K+1}(t)dt+\sqrt{2\lambda}dW(t),$$
where $W(t)$ is the standard Brownian motion,
provided $\bar{Q}_{K+1}^N(0) \dto  \bar{Q}_{K+1}(0)$ in $\mathbbm{R}$.
\item For $i\geq K+2$, $\big\{\bar{Q}^N_i(t)\big\}_{t\geq 0}\dto \big\{\bar{Q}_i(t)\big\}_{t\geq 0}$, where $\bar{Q}_i(t)\equiv 0$, provided $\bar{Q}_i^N(0)\pto 0$.
\end{enumerate}
\end{theorem}
Note that statements~(i) and~(ii) in Theorem~\ref{th:diffusion} imply statement~(i) in Theorem~\ref{th:diff pwr of d 1}, for the JSQ policy, while (iii) and (iv) in Theorem~\ref{th:diffusion} are equivalent with statements (ii) and (iii) in Theorem~\ref{th:diff pwr of d 1}.
In view of the universality result in Corollary~\ref{cor-diff}, it thus suffices to prove Theorem~\ref{th:diffusion}.

The rest of this section is devoted to the proof of Theorem~\ref{th:diffusion}.
From a high level, the idea of the proof is the following. 
Introduce 
\begin{equation}\label{eq:pos-neg}
\begin{split}
Y^N(t):=\sum_{i=1}^B Q_i^N(t),\qquad
D^N_+(t):=\sum_{i=1}^{K}(N-Q_i^N(t)),\qquad
D^N_-(t):=\sum_{i=K+2}^{B}Q_i^N(t).
\end{split}
\end{equation}
and observe that
\begin{align*}
Q^N_{K+1}(t)+KN &= \sum_{i=1}^B Q_i^N(t)+\sum_{i=1}^{K} (N-Q_i^N(t)) - \sum_{i=K+2}^B Q_i^N(t)\\
&= Y^N(t) + D^N_+(t) -D^N_-(t).
\end{align*}
We show in Proposition~\ref{prop:positive}
that the sequence of processes $\big\{D^N_+(t)\big\}_{t\geq 0}$ is $\Op(\log(N))$, which implies
 that the number of server pools with fewer than $K$ active tasks is negligible on $\sqrt{N}$-scale.
 Furthermore, in Proposition~\ref{prop:negative} we prove that
since $\lambda<B$ the number of  tasks that are assigned to  server pools with at least $K+1$ tasks converges to zero in probability
 and hence, for a suitable starting state, $\big\{D^N_-(t)\big\}_{t\geq 0}$ converges 
 to the zero process.
 As we will show, this also means that $Y^N(t)$ behaves with high
 probability as the total number of tasks in an M/M/$\infty$ system. 
Therefore with the help of the following diffusion limit result for the M/M/$\infty$ system in~\cite[Theorem 6.14]{Robert03}, we conclude the proof of statement~(iii) of Theorem~\ref{th:diffusion}.
\begin{theorem}[{\cite[Theorem 6.14]{Robert03}}]
\label{th:robert-book-mmn}
Let $\big\{Y^N_\infty(t)\big\}_{t\geq 0}$ be the total number of tasks in an M/M/$\infty$ system with arrival rate $\lambda (N)$ and unit-mean service time. 
If $(Y^N_\infty(0)-\lambda (N))/\sqrt{N}\to v\in\R$, then the process $\big\{\bar{Y}^N_{\infty}(t)\big\}_{t\geq 0}$, with
$$\bar{Y}^N_{\infty}(t)=\frac{Y^N_\infty(t)-\lambda( N)}{\sqrt{N}},$$
converges weakly to an Ornstein-Uhlenbeck process $\big\{X(t)\big\}_{t\geq 0}$ 
described by the stochastic differential equation
$$X(0)=v,\qquad \dif X(t) = -X(t)\dif t + \sqrt{2\lambda}\dif W(t).$$
\end{theorem}
The next two propositions state asymptotic properties of $\big\{D^N_+(t)\big\}_{t\geq 0}$ and $\big\{D_-^N(t)\big\}_{t\geq 0}$ mentioned before, which play a crucial role in the proof of Theorem~\ref{th:diffusion}. 
Let $B_{K+1}^N(t)$ be the cumulative number of tasks up to time $t$ that are assigned to some server pool having at least $K+1$ active tasks if $B>K+1$, and that are lost if $B=K+1$.
\begin{proposition}\label{prop:negative}
Under the assumptions of Theorem~\ref{th:diffusion}, for any $T\geq 0$, $B_{K+1}^N(T)\pto 0$, and consequently,
$\sup_{t\in[0,T]}D_-^N(t)\pto 0$ as $N\to\infty,$ 
provided $D_-^N(0)\pto 0$.
\end{proposition}
Informally speaking, the above proposition implies that for large $N$, there will be almost no server pool with $K+2$ or more tasks in any finite time horizon, if the system starts with no server pools with more than $K+1$ tasks. 
The next proposition shows that the number of server pools having fewer than $K$ tasks is of order $\log (N)$ in any finite time horizon.

\begin{proposition}\label{prop:positive}
Under the assumptions of Theorem~\ref{th:diffusion}, the sequence $\left\{D^N_+(t)/\log (N)\right\}_{t\geq 0}$ is stochastically bounded in $D_{\mathbbm{R}}[0,\infty)$, provided $\left\{D^N_+(0)/\log (N)\right\}_{N\geq 1}$ is a tight sequence of random variables.
\end{proposition}
Before providing the proofs of the above two propositions, we first prove Theorem~\ref{th:diffusion} using Propositions~\ref{prop:negative} and~\ref{prop:positive}.
\begin{proof}[Proof of Theorem~\ref{th:diffusion}]
First observe that (iv) and (ii) immediately follows from Propositions~\ref{prop:negative} and~\ref{prop:positive}, respectively.

To prove (i), fix any $T\geq 0$.
We will show that
\begin{equation}
\lim_{N\to\infty}\Pro{\sup_{t\in[0,T]}\sum_{i=1}^{K-1}\bQ_{i}^N(t)\leq 1}=1.
\end{equation} 
 Since $\bQ^N_{i}\leq 1$ implies that $\bQ^N_{i-1}\leq 1$ for $i= 2,\ldots,K$, this then completes the proof of~(i).
Note that the process $\sum_{i=1}^{K-1}\bQ_{i}^N(\cdot)$ increases by one when there is a departure from some server pool with at most $K-1$ active tasks, and if positive, decreases by one whenever there is an arrival.
Therefore it can be thought of as a
birth-death process with state-dependent instantaneous birth rate $\sum_{i=1}^{K-1}i(Q^N_i(t)-Q^N_{i+1}(t))$, and constant instantaneous death rate $\lambda (N)$. 
Since
\begin{align*}
\sum_{i=1}^{K-1}i(Q^N_i(t)-Q^N_{i+1}(t))&=\sum_{i=1}^{K-1}Q^N_i(t)-(K-1)Q^N_{K}(t)\leq (K-1)(N-Q_K^N(t)),
\end{align*}
the process $\big\{\sum_{i=1}^{K-1}\bQ_{i}^N(t)\big\}_{t\geq 0}$ 
is stochastically upper bounded by a birth-and-death process $\big\{Z^N(t)\big\}_{t\geq 0}$ with birth rate $(K-1)(N-Q_{K}^N(t))$ and constant death rate $\lambda( N).$
Due to (ii), we can claim that for \textit{any} nonnegative sequence $\ell(N)$ diverging to infinity,
$$\lim_{N\to\infty}\Pro{\sup_{t\in[0,T]}(N-Q^N_K(t))\leq \ell (N)\log (N)}=1.$$
Let $\big\{\eta^N(n)\big\}_{n\geq 1}$ denote the discrete uniformized chain of the upper bounding birth-death process. 
Also, let $K_N(t)$ denote the number of jumps taken up to time $t$ by $\big\{\eta^N(n)\big\}_{n\geq 1}$. 
Since the jump rate of the process is $O(N)$, we have for \textit{any} nonnegative sequence $\ell^0(N)$ diverging to infinity, and for any $T\geq 0$, 
$$\lim_{N\to\infty}\Pro{K_N(T)\leq N\ell^0(N)}=1.$$
Given $Q_K^N$, considering the $\eta^N(\cdot)$ Markov chain, the probability of one birth is bounded from above by 
$$p_{Q_K^N}=\frac{(K-1)(N-Q_K^N)}{N+(K-1)(N-Q^N_{K})}.$$
Now, $Z^N(\cdot)$ will exceed 1 if and only if there are at least two successive births. Hence,
\begin{equation}\label{eq:split}
\begin{split}
&\Pro{\sup_{t\in[0,T]}Z^N(t)\leq 1}=\Pro{\sup_{n\leq K_N(T)}\eta^N(n)\leq 1}\\
&\geq \Pro{\sup_{n\leq N\ell^0(N)}\eta^N(n)\leq 1}\Pro{K_N(T)\leq N\ell^0(N)}+\Pro{K_N(T)> N\ell^0(N)}.
\end{split}
\end{equation}
Again we can write the first term of the last inequality above as
\begin{align*}
&\Pro{\sup_{n\leq N\ell^0(N)}\eta^N(n)\leq 1}\\
&\geq \Pro{\sup_{n\leq N\ell^0(N)}\eta^N(n)\leq 1\given \sup_{t\in[0,T]}(N-Q^N_K(t))\leq  \ell(N)\log (N)}\\
&\hspace{6.7cm}\times\Pro{\sup_{t\in[0,T]}(N-Q^N_K(t))\leq \ell(N)\log (N)}\\
&\geq \left(1-\left(\frac{(K-1)\ell(N)\log (N)}{N+(K-1)\ell(N)\log (N)}\right)^2\right)^{N\ell^0(N)}\times\Pro{\sup_{t\in[0,T]}(N-Q^N_K(t))\leq \ell(N)\log (N) }.
\end{align*}
If we choose $\ell(N)$ and $\ell^0(N)$ such that $\ell(N)^2\ell^0(N)\log (N)/ N\to 0$ as $N\to\infty$, then the expression on the right of~\eqref{eq:split} converges to 1 (one can see that this choice is always feasible).
 Hence the proof of (i) is complete.

For (iii), recall that $Y^N_\infty(t)$ denotes the total number of tasks in an M/M/$\infty$ system with arrival rate $\lambda(N)$ and exponential service time distribution with unit mean.
Also, Proposition~\ref{prop:negative} implies that under the assumptions of the theorem, in any finite time horizon, with high probability there will be no arrival to a server pool with $K+1$ or more active tasks.
Now observe that since $B\geq K+1$, for any $T\geq 0$,
\begin{align*}
\Pro{\exists\ t\in[0,T]: Y^N(t) \neq Y^N_\infty(t)}
\leq \Pro{\exists\ t\in[0,T]: B^N_{K+1}(t)\geq 1}\to 0,\quad\text{as}\quad N\to\infty.
\end{align*}
Propositions~\ref{prop:negative} and~\ref{prop:positive} then yield
\begin{align*}
&\sup_{t\in[0,T]}\frac{1}{\sqrt{N}}\left|Q_{K+1}^N(t)-f(N)-(Y^N_\infty(t)-\lambda (N))\right|\\
= & \sup_{t\in[0,T]}\frac{1}{\sqrt{N}}\left|\sum_{i=1}^B Q_{i}^N(t)+\sum_{i=1}^K(N- Q_{i}^N(t))-\sum_{i=K+2}^BQ^N_i(t)-KN-f(N)-(Y^N_\infty(t)-\lambda (N))\right| \\
= & \sup_{t\in[0,T]}\frac{1}{\sqrt{N}}\left[Y^N(t)-Y^N_\infty(t)
+D_N^+(t)-D^N_-(t)\right]\to 0,
\end{align*}
as $N\to\infty$,
which in conjunction with \cite[Theorem 6.14]{Robert03}, as mentioned earlier, gives the desired diffusion limit.
\end{proof}

\begin{proof}[Proof of Proposition~\ref{prop:negative}]
Couple the M/M/$\infty$ system and a system under the ordinary JSQ policy in the natural way, until an overflow event occurs in the latter system.
Observe that for any fixed $M>0$, the event $\left[\sup_{t\in[0,T]}B^N_{K+1}(t)\geq M\right]$ will occur only if for some $t'\leq T$,
some arriving task is assigned to a server pool with more than $K$ active tasks,
and in that case, there exists $t''\leq t'$, such that $Y^N(t'')> (\lambda+\varepsilon)N$, for some $\varepsilon>0$ with $\lambda+\varepsilon<1$.
Since, for any $t\in[0,t'']$, $Y^N(t) = Y^N_\infty(t)$, 
we have
\begin{equation}\label{eq:neg_part}
\begin{split}
&\sup_{t\in[0,T]}B^N_{K+1}(t)\geq M\\
&\implies\sup_{t''\in[0,t']}Y^N(t'')\geq  (\lambda+\varepsilon)N\\
&\implies\sup_{t''\in[0,t']}Y^N_\infty(t'')\geq  (\lambda+\varepsilon)N\\
&\implies\sup_{t\in[0,T]}(Y^N_\infty(t)-\lambda( N))> \varepsilon N +o(N)\\
&\implies\sup_{t\in[0,T]}\frac{1}{\sqrt{N}}(Y^N_\infty(t)-\lambda (N))> \varepsilon\sqrt{N}+o(\sqrt{N}).
\end{split}
\end{equation}
From Theorem 6.14 of \cite{Robert03}, we know that the process $\big\{(Y^N(t)-\lambda (N))/\sqrt{N}\big\}_{t\geq 0}$ is stochastically bounded. 
Hence, Equation~\eqref{eq:neg_part} yields that for any $T\geq 0$, $\sup_{t\in[0,T]}B^N_{K+1}(t)$ converges to zero in probability as $N\to\infty$.
Consequently, from the assumption of Theorem~\ref{th:diffusion} that $D^N_-(0)\pto 0$, the conclusion $\sup_{t\in[0,T]}D^N_-(t)\pto 0$, is immediate.
\end{proof}

\begin{proof}[Proof of Proposition~\ref{prop:positive}]
Observe that $\sum_{i=1}^K(N-Q_i^N(\cdot))$ increases by one when there is a departure from some server pool with at most $K$ active tasks, and if positive, decreases by one whenever there is an arrival.
Therefore the process $\big\{D^N_+(t)\big\}_{t\geq 0}$ increases by one at rate 
$\sum_{i=1}^K i(Q_i(t)-Q_{i+1}(t))=\sum_{i=1}^K (Q_i(t)-Q_{K+1}(t))$, and while positive, decreases by one at  constant rate $\lambda(N)$. 
Now, to prove stochastic boundedness of the sequence of processes $\big\{D^N_+(t)/\log (N)\big\}_{t\geq 0}$, we will show that for any fixed $T\geq 0$ and any function $\ell(N)$ diverging to infinity (i.e., such that $\ell(N)\to\infty$ as $N\to\infty$), 
\begin{equation}\label{eq:conv_prob}
\Pro{\sup_{t\in[0,T]}D^N_+(t)>\ell(N)\log (N)}\to 0.
\end{equation}

Let $\big\{X^N(n)\big\}_{n\geq 0}$ be the discrete jump chain, and $K_N(t)$ be the number of jumps before time $t$, of the process $\big\{D_+^N(t)\big\}_{t\geq 0}$. Hence, for any fixed $T\geq 0$,
\begin{equation}\label{eq:condition}
\begin{split}
&\Pro{\sup_{t\in[0,T]}D^N_+(t)>\ell(N)\log (N)}\\
&=\Pro{\sup_{n\leq K_N(T)}X^N(n)>\ell(N)\log (N)}\\
&\leq \Pro{\sup_{n\leq N\ell_0(N)}X^N(n)>\ell(N)\log (N)}\Pro{K_N(T)\leq N\ell_0(N)}\\
&+\Pro{K_N(T)> N\ell_0(N)},
\end{split}
\end{equation}
for some function $\ell_0(N):\mathbbm{N}\to\mathbbm{N}$, to be chosen according to Lemma~\ref{lem:discr-walk} below. 
Now, observe that $K_N(T)$ is upper bounded by a Poisson random variable with parameter $\lambda(N)T+\int_0^T \sum_{i=1}^K (Q_i(s)-Q_{K+1}(s))ds$, and $\sum_{i=1}^K (Q_i(s)-Q_{K+1}(s))\leq KN$. 
Hence for any function $\ell_0(N)$ diverging to infinity, we have 
$$\Pro{K_N(T)> N\ell_0(N)}\to 0.$$

To control the first term, it is enough to note that $\sum_{i=1}^K (Q_i(t)-Q_{K+1}(t))\leq KN<\lambda N$. Hence the process $\big\{X^N(n)\big\}_{n\geq 1}$ can be stochastically upper bounded by the process $\big\{\hat{X}^N(n)\big\}_{n\geq 1}$, defined as follows:
\begin{equation}\label{eq:upperboundingbd}
\hat{X}^N(n+1)=
\begin{cases}
\hat{X}^N(n)+1&\mbox{ with prob. }K/(K+\lambda)\\
(\hat{X}^N(n)-1)\vee 0&\mbox{ with prob. }\lambda/(K+\lambda)
\end{cases}
\end{equation}
Therefore combining Lemma~\ref{lem:discr-walk} below for the above Markov process $\big\{\hat{X}^N(n)\big\}_{n\geq 0}$ with Equation~\eqref{eq:condition} we obtain Equation~\eqref{eq:conv_prob}. Hence the proof is complete.
\end{proof}
\begin{lemma}\label{lem:discr-walk}
For any function $\ell(N):\mathbbm{N}\to\mathbbm{N}$, diverging to infinity, there exists another function $\ell_0(N):\mathbbm{N}\to\mathbbm{N}$, diverging to infinity, such that 
$$\Pro{\sup_{n\leq N\ell_0(N)}\hat{X}^N(n)>\ell(N)\log (N)}\to 0.$$
\end{lemma}
\begin{proof}
We will use a regenerative approach to prove the lemma. Let $p:=K/(K+\lambda)$. 
Note that then $p<q:=1-p$. 
Define the $i^{\mathrm{th}}$ regeneration time $\rho_i$ of the Markov chain as follows: $\rho_0=0$, and $\rho_i:=\min\big\{k>\rho_{i-1}:\hat{X}_k=0\big\}$, for $i\geq 1$. 
Also define, $m_i:=\max\big\{\hat{X}_k: \rho_{i-1}\leq k<\rho_i\big\}$, for $i\geq 1$, and $\xi(n):=\min\big\{i:\rho_i\geq n\big\}$, for $n\geq 1$. 
Now observe that \cite[XIV.2]{feller1}, 
\begin{equation}
\Pro{m_i\geq M}=p\times\frac{\frac{q}{p}-1}{\left(\frac{q}{p}\right)^M-1}\leq a^{-M},
\end{equation}
for some $a>1$, since $q/p>1$. 
Thus the tail of the distribution of the maximum attained in one regeneration period decays exponentially. Recall that, in $n$ steps the Markov chain exhibits $\xi(n)$ regenerations. Hence, for any $\ell_0(N)$ and $\ell(N)$,
\begin{equation}\label{eq:discrete-maximum}
\begin{split}
&\Pro{\sup_{n\leq N\ell_0(N)}\hat{X}^N(n)>\ell(N)\log (N)}=\Pro{\sup_{i\leq \xi(N\ell_0(N))}m_i>\ell(N)\log (N)}\\
&\leq 1-\left(1-a^{-\ell(N)\log (N)}\right)^{\xi(N\ell_0(N))}\leq 1-\left(1-a^{-\ell(N)\log (N)}\right)^{N\ell_0(N)}.
\end{split}
\end{equation}

Now, for given $\ell(N)$, choose $\ell_0(N)$ diverging to infinity, such that 
$$N\ell_0(N) a^{-\ell(N)\log (N)}\to 0\quad \text{as}\quad N\to\infty.$$ 
Since the condition is equivalent to 
$$\log (N)+\log(\ell_0(N))-\ell(N)\log( a)\log (N)\to-\infty,$$ 
it is evident that such a choice of $\ell_0(N)$ is always possible. Hence, for such a choice of $\ell_0(N)$ the probability in Equation~\eqref{eq:discrete-maximum} converges to zero and the proof is complete.
\end{proof}

\section{Diffusion Limit of JSQ: Integral \texorpdfstring{$\boldsymbol{\lambda}$}{lambda}}\label{sec:integral}
In this section we analyze the diffusion-scale behavior of the ordinary JSQ policy when
$\lambda$ is an integer, i.e., $f=0$, and
$$\frac{KN-\lambda(N)}{\sqrt{N}}\to\beta,\quad \text{as}\quad N\to\infty,$$ 
with $\beta\in\R$ being a fixed real number. 
Throughout this section we assume $B=K+1.$
Thus, tasks that arrive when all the server pools have $K+1$ active tasks, are permanently discarded. 
For brevity in notation, define, $Z^N_1(t)=\sum_{i=1}^K (N-Q_i^N(t))$ and $Z^N_2(t):=Q_{K+1}^N(t)$. 
Note that $Z^N_1(t)$ corresponds to $D^N_+(t)$ in the previous section.
Also recall \eqref{eq:scaling-f=0}, and define
\begin{equation}
\begin{split}
\zeta_1^N(t)&:=\frac{Z_1^N(t)}{\sqrt{N}}=\hQ_{K-1}^N(t)+\hQ_{K}^N(t)\\
\zeta_2^N(t)&:=\frac{Z_2^N(t)}{\sqrt{N}}=\hQ_{K+1}^N(t),
\end{split}
\end{equation}
with $\hQ_{K-1}^N(t)$, $\hQ_{K}^N(t)$, and $\hQ_{K+1}^N(t)$ as in~\eqref{eq:scaling-f=0}.

\begin{theorem}\label{th: f=0 diffusion}
Assume that $(\zeta^N_1(0),\zeta^N_2(0))\dto (\zeta_1(0),\zeta_2(0))$ in $\R^2$ as $N\to\infty$. 
Then the two-dimensional process $\big\{(\zeta^N_1(t),\zeta^N_2(t))\big\}_{t\geq 0}$ converges weakly to the process $\big\{(\zeta_1(t),\zeta_2(t))\big\}_{t\geq 0}$ in $D_{\R^2}[0,\infty)$ governed by the following stochastic recursion equation:
\begin{align*}
\zeta_1(t)&=\zeta_1(0)+\sqrt{2K}W(t)-\int_0^t(\zeta_1(s)+K\zeta_2(s))\dif s+\beta t+V_1(t)\\
\zeta_2(t)&=\zeta_2(0)+V_1(t)-(K+1)\int_0^t\zeta_2(s)\dif s,
\end{align*}
where $W$ is the standard Brownian motion, and $V_1(t)$ is the unique non-decreasing process in $D_{\R_+}[0,\infty)$ satisfying
$$\int_0^t\ind{\zeta_1(s)\geq 0}\dif V_1(s)=0.$$
\end{theorem}

\begin{remark}{\normalfont
Note that $Y^N(t)-KN = Z_2^N(t) - Z_1^N(t)$. Thus, under the assumption in~\eqref{eq:f=0}, the diffusion limit in Theorem~\ref{th: f=0 diffusion} implies that 
\begin{align*}
\frac{Y^N(\cdot)-\lambda(N)}{\sqrt{N}} = \frac{Y^N(\cdot)-KN}{\sqrt{N}} + \frac{KN-\lambda(N)}{\sqrt{N}}\dto \zeta_2(\cdot)-\zeta_1(\cdot) + \beta .
\end{align*}
Writing $X(t) = \zeta_2(t) -\zeta_1(t) -\beta$, from Theorem~\ref{th: f=0 diffusion}, one can note that the process $\big\{X(t)\big\}_{t\geq 0}$ satisfies 
$$\dif X(t) = -X(t)\dif t - \sqrt{2K}\dif W(t),$$
which is consistent with the diffusion-level behavior of $Y^N(\cdot)$ stated in Theorem~\ref{th:robert-book-mmn}. 
}
\end{remark}
Next, using the arguments in the proof of Proposition~\ref{prop:positive} one can see that the process
$$\sum_{i=1}^{K-1}\frac{N-Q_i^N(\cdot)}{\sqrt{N}}=\hQ^N_{K-1}(\cdot)\pto 0,$$
provided  $\hQ_{K-1}^N(0)\pto 0$.
Thus, Theorem~\ref{th: f=0 diffusion} yields the diffusion limit for the ordinary JSQ policy in the case $B=K+1$.
The proof for $B>K+1$ then follows from exactly the same arguments as provided in~\cite[Section~5.2]{EG15}.
The idea is that since the process $Q_{K+1}^N(\cdot)$, when scaled by $\sqrt{N}$, is stochastically bounded, the probability that on any finite time interval, it will take value $N$ (or equivalently, all server pools will have at least $K+1$ active tasks) vanishes as $N$ grows large.
Therefore, the dynamics of the limit of $(\hQ_{K+2}^N(\cdot),\ldots,\hQ_M(\cdot))$ becomes deterministic, and the limit of $\hQ^N_{K+1}(\cdot)$ for $B>K+1$ becomes a transformation of the limit of $\hQ^N_{K+1}(\cdot)$ for $B=K+1$, as described in Theorem~\ref{th:diff pwr of d 2}. 
Hence, note that the diffusion limit in Theorem~\ref{th: f=0 diffusion} is equivalent to the one in Theorem~\ref{th:diff pwr of d 2}.
In view of the universality result in Corollary~\ref{cor-diff}, it thus suffices to prove Theorem~\ref{th: f=0 diffusion}.

We will use the reflection argument developed in \cite{EG15} to prove Theorem~\ref{th: f=0 diffusion}. 
Observe that the evolution of $\big\{(Z_1^N(t),Z_2^N(t))\big\}_{t\geq 0}$ can be described by the following stochastic recursion which is explained in detail below.
\begin{equation}\label{eq: f=0 process}
\begin{split}
Z_1^N(t)&=Z_1^N(0)+A_1\left(\int_0^t(KN-Z_1^N(s)-KZ_2^N(s))\dif s\right)
-D_1(\lambda(N) t)+U_1^N(t)\\
Z_2^N(t)&=Z_2^N(0)+U_1^N(t)-D_2\left(\int_0^t(K+1)Z_2^N(t)\dif s\right)-U_2^N(t),
\end{split}
\end{equation}
where $A_1$, $D_1$ and $D_2$ are unit-rate Poisson processes, and
\begin{equation}
\begin{split}
U_1^N(t)&=\int_0^t\ind{Z_1^N(s)=0}\dif D_1(\lambda(N)s)\\
U_2^N(t)&=\int_0^t\ind{Z_2^N(s)=C\sqrt{N}}\dif D_1(\lambda(N)s).
\end{split}
\end{equation}

The components of Equation~\eqref{eq: f=0 process} can be explained as follows. 
The process $Z_1(t)$ increases by one when a departure occurs from a server pool with  at most $K$ active tasks, and it decreases by one when an arriving task is assigned to a server pool with at most $K$ active tasks. 
Hence the instantaneous rate of increase at time $s$ is given by 
\begin{align*}
\sum_{i=1}^{K}i(Q^N_i(t)-Q^N_{i+1}(t))&=\sum_{i=1}^{K}Q^N_i(t)-KQ^N_{K+1}(t)\\
&=KN - \sum_{i=1}^{K}(N-Q^N_i(t))-KQ^N_{K+1}(t) \\
&= KN-Z_1^N(t)-KZ_2^N(t),
\end{align*}
and the instantaneous rate of decrease is given by the arrival rate $\lambda(N)$. 
But $Z_1^N$ cannot be negative, and hence the arrivals when $Z_1^N$ is zero, add to $Z_2^N$, and the rate of increase of the $Z_2^N$ process is given by the overflow process $U_1^N$. 
Since $B=K+1$, the rate of decrease of $Z_2^N$ equals the total number of tasks at server pools with exactly $K+1$ tasks, which is given by $(K+1)Z_2^N$. 
This explains the rate in the Poisson process $D_2(\cdot)$. 
Finally, since $Z_2^N$ is upper bounded by $N$, $U_2^N$ is the overflow of the $Z_2^N$ process with $C=\sqrt{N}$, i.e., the number of arrivals to the system when $Z_2^N=N$. 
The existence and uniqueness of the above stochastic recursion can be proved following the arguments in \cite[Section 2]{PTRW07}.

\paragraph{Martingale representation}
We now introduce the martingale representation for~\eqref{eq: f=0 process}, and following similar arguments as in \cite[Subsection 4.3]{EG15}, we obtain the following scaled, square integrable martingales with appropriate filtration:
\begin{equation}\label{eq:martingales1}
\begin{split}
M^N_{1,1}(t)&=\frac{1}{\sqrt{N}}A_1\left(\int_0^t(KN-Z_1^N(s)-KZ_2^N(s))\dif s\right)-\frac{1}{\sqrt{N}}\int_0^t(KN-Z_1^N(s)-KZ_2^N(s))\dif s\\
M_{1,2}^N(t)&=\frac{1}{\sqrt{N}}(D_1(\lambda(N) t)-\lambda(N)t)\\
M^N_{2,1}(t)&=\frac{1}{\sqrt{N}}D_2\left(\int_0^t(K+1)Z_2^N(t)\dif s\right)-\frac{K+1}{\sqrt{N}}\int_0^tZ_2^N(s)\dif s,
\end{split}
\end{equation}
with $V_1^N(t):=U_1^N(t)/\sqrt{N}$ and $V_2^N(t):=U_2^N(t)/\sqrt{N}$, and the predictable quadratic variation processes given by
\begin{equation}
\begin{split}
\langle M^N_{1,1}\rangle(t)&=\frac{1}{N}\int_0^t(KN-Z_1^N(s)-KZ_2^N(s))\dif s\\
\langle M_{1,2}^N\rangle(t)&=\frac{\lambda(N)t}{N}\\
\langle M^N_{2,1}\rangle(t)&=\frac{K+1}{N}\int_0^tZ_2^N(s)\dif s.
\end{split}
\end{equation}

Therefore, we have the following martingale representation for~\eqref{eq: f=0 process}:
\begin{equation}\label{eq:martingale rep}
\begin{split}
\zeta^N_1(t)&=\zeta^N_1(0)+M_{1,1}^N(t)-M_{1,2}^N(t)
-\int_0^t(\zeta_1^N(s)+K\zeta_2^N(s))\dif s+\frac{t(KN-\lambda(N))}{\sqrt{N}}+V_1^N(t)\\
\zeta_2^N(t)&=\zeta_2^N(0)+V_1^N(t)-M_{2,1}^N(t)
-(K+1)\int_0^t\zeta_2^N(s)\dif s -V_2^N(t)
\end{split}
\end{equation}

\paragraph{Convergence of independent martingales}
We now show the convergence of the martingales defined in~\eqref{eq:martingales1} using the functional central limit theorem.
\begin{lemma}\label{lem:martingale convergence}
As $N\to\infty$,
$$\left\{\left(M^N_{1,1}(t), M^N_{1,2}(t), M^N_{2,1}(t)\right)\right\}_{ t\geq 0}\dto \left\{\left(\sqrt{K} W_1(t),\sqrt{K} W_2(t), 0\right)\right\}_{t\geq 0}$$ in $D_{\R^3}[0,\infty),$  where $W_1$, $W_2$ are independent standard Brownian motions.
\end{lemma}
\begin{proof}
From Theorem~\ref{fluidjsqd} we know that for any fixed $T\geq 0$, 
$$\sup_{t\in [0,T]}Z_1^N(t)/N\pto 0 \quad\text{and}\quad\sup_{t\in [0,T]}Z_2^N(t)/N\pto 0.$$
This yields the following convergence results: 
\begin{equation}
\begin{split}
\langle M^N_{1,1}\rangle(T)&\pto KT\\
\langle M_{1,2}^N\rangle(T)&\pto\lambda T=KT\\
\langle M^N_{2,1}\rangle(T)&\pto 0.
\end{split}
\end{equation}
Then, using a random time change, the continuous-mapping theorem and functional central limit theorem \cite[Theorem 4.2]{PTRW07}, \cite[Lemma 6]{EG15}, we get the convergence of the martingales.
\end{proof}
Now we use the continuous-mapping theorem to prove the convergence of the processes described in~\eqref{eq:martingale rep}. To proceed in that direction, we need the following proposition, which is analogous to \cite[Lemma 1]{EG15}.

\begin{proposition}\label{th:continuous}
Let $B\in\bar{\R}_+$, $b\in\R^2$, $(y_1,y_2)\in D^2[0,\infty)$, and $(x_1,x_2)\in D^2[0,\infty)$ be defined by the following recursion: for $t\geq 0$,
\begin{equation}\label{eq:recursion1}
\begin{split}
x_1(t)&=b_1+y_1(t)+\int_0^t(-x_1(s)- Kx_2(s))\dif s+u_1(t)\\
x_2(t)&=b_2+y_2(t)+(K+1)\int_0^t(-x_2(s))\dif s+u_1(t)-u_2(t),
\end{split}
\end{equation}
where $u_1$ and $u_2$ are unique non-decreasing functions in $D$, such that
\begin{equation}\label{eq:reflection}
\begin{split}
\int_0^{\infty}\ind{x_1(s)>0}\dif u_1(t) & =0\\
\int_0^{\infty}\ind{x_2(s)<B}\dif u_2(t) & =0.
\end{split}
\end{equation}
Then, $(x,u)$ is the unique solution to the above system. 
Furthermore, there exist functions $(f,g):(\bar{\R},\R^2, D^2_\R[0,\infty))\to (D^2_\R[0,\infty),D^2_\R[0,\infty))$ with $x=f(B,b,y)$ and $u=g(B,b,y)$, which are continuous when $\bar{\R}_+$ is equipped with order topology, $D_\R[0,\infty)$ is equipped with topology of uniform convergence over compact sets, and $(\bar{\R},\R^2, D^2_\R[0,\infty))$ and $(D^2_\R[0,\infty),D^2_\R[0,\infty))$ are equipped with product topology.
\end{proposition}
The proof of the above proposition follows from similar arguments as described in the proof of \cite[Lemma 1]{EG15}, and hence is omitted.
\begin{proof}[Proof of Theorem~\ref{th: f=0 diffusion}]
Observe that the stochastic recursion equations described by \eqref{eq:martingale rep} fit in the framework of the recursion described by \eqref{eq:recursion1}, by taking $b_i=\zeta_i^N(0)$, $i=1,2$, $C=\sqrt{N}$, $y_1(t)=M^N_{1,1}(t)-M^N_{1,2}(t)+t(KN-\lambda(N))/\sqrt{N}$, and $y_2(t)=-M^N_{2,1}(t)$ for the $N^{\mathrm{th}}$ process.

By the assumptions of the theorem we have $\zeta_i^N(0)\dto \zeta_i(0)$, for $i=1,2$. Also, by Lemma~\ref{lem:martingale convergence}, $\big\{(M^N_{1,1}(t), M^N_{1,2}(t), M^N_{2,1}(t))\big\}_{t\geq 0}\dto \big\{(\sqrt{K} W_1(t), \sqrt{K} W_2(t), 0)\big\}_{t\geq 0}$. Hence, for the limiting process, $y_1(t)=\sqrt{K} W_1(t)-\sqrt{K} W_2(t)+\beta t\equiv \sqrt{2K}W(t)+\beta t$ and $y_2(t)\equiv 0$.
Finally, using the continuous-mapping theorem we get the desired convergence as in the proof of~\cite[Theorem~2]{EG15}.
\end{proof}

\section{Performance Implications}\label{sec:performance}

\subsection{Evolution of number of tasks at tagged server pool}
We now provide some insights into the steady-state dynamics of the number of tasks at a particular server pool in the regime $d(N)\to\infty$ as $N\to\infty$.
Due to exchangeability of the server pools, asymptotically, the dynamics at a particular server pool depends on the system only through the mean-field limit, or the global system state averages.
Based on the fixed point~\eqref{eq:fixed point}, we claim (without proof) that the steady-state dynamics can be described as follows:
\begin{enumerate}[{\normalfont (i)}]
\item If a server pool contains $\lceil\lambda\rceil$ active tasks, then with high probability no further task will be assigned to it. 
\item Similarly, if a departure occurs from a server pool having $K=\lfloor\lambda\rfloor$ active tasks, a task will immediately be assigned to it.
\item Since the total flow of arrivals that join server pools with exactly $K$ active tasks, are
distributed uniformly among all such server pools, each server pool with exactly $K$ active tasks will observe an arrival rate $\lambda p_K(\qq^\star)/(q^\star_K-q^\star_{K+1}) = (K+1)f/(1-f)$.
\item Finally, the rate of departure from a server pool with $K+1$ active tasks is given by $K+1$.
\end{enumerate}
Let $S^{\sss d(N)}_k(t)$ denote the number of tasks at server pool~$k$ at time $t$ in the $N^{\mathrm{th}}$ system under the JSQ$(d(N))$ scheme.
Combining all the above, provided $d(N)\to\infty$ as $N\to\infty$, the process $\big\{S^{\sss d(N)}_k(t)\big\}_{t\geq 0}$ converges in distribution to the process $\big\{S(t)\big\}_{t\geq 0}$, described as follows:
\begin{enumerate}[{\normalfont (i)}]
\item If $f>0$, then $\big\{S(t)\big\}_{t\geq 0}$ is a two-state process, taking values $K$ and $K+1$, with transition rate from $K$ to $K+1$ given by $(K+1)f/(1-f)$, and from $K+1$ to $K$ given by $K+1$. 
So the steady-state distribution is $\Pro{S=K}=1-f$, and $\Pro{S=K+1}=f$, i.e., for $i\geq 1$, $\Pro{S=i}=q_i^\star-q_{i+1}^\star$, which agrees with the fixed point~\eqref{eq:fixed point} of the fluid limit.
\item If $f=0$, then $\big\{S(t)\big\}_{t\geq 0}$ is a constant process, taking value $\lambda=K$.
\end{enumerate}

\subsection{Evolution of number of tasks observed by tagged task}
To analyze the performance perceived by a particular tagged task with execution time $T$, observe that in steady state the probability that it will join a server pool with $i$ active tasks is given by
$p_i(\qq^\star)=K(1-f)/\lambda$ for $i=K-1$, $(K+1)f/\lambda$ for $i=K$, and 0 otherwise.
In the time interval $[0,T]$, the number of active tasks in the server pool it joins, is again a birth-death process $\big\{\hat{S}(t)\big\}_{0\leq t\leq T}$, whose dynamics is the same as of $\big\{S(t)\big\}_{t\geq 0}$ process conditioned on having one permanent task (i.e., its departure is not allowed).
Therefore, $\big\{\hat{S}(t)\big\}_{0\leq t\leq T}$ can be described as follows:
\begin{enumerate}[{\normalfont (i)}]
\item If $f>0$, then $\big\{\hat{S}(t)\big\}_{0\leq t\leq T}$ is a two-state process, taking values $K$ and $K+1$, with transition rate from $K$ to $K+1$ given by $(K+1)f/(1-f)$, and from $K+1$ to $K$ given by~$K$.
The steady-state distribution of the process is then given by
$\Pro{\hat{S}=K}=K(1-f)/\lambda$, and $\Pro{\hat{S}=K+1}=(K+1)f/\lambda$.
\item If $f=0$, then $\big\{\hat{S}(t)\big\}_{0\leq t\leq T}$ is a constant process, taking value $\lambda=K$.
\end{enumerate}
Observe in both of the above two cases that the initial distribution of $\hat{S}(t)$ coincides with its stationary distribution.
Now, if the performance perceived by the tagged task is measured as a function $h:\N\to \R$ of the number of concurrent tasks, then the relevant performance measure is given by 
\begin{align}
\E{\frac{1}{T}\int_0^T h(\hat{S}(t))\dif t}=\frac{1}{\lambda}((1-f)Kh(K)+f(K+1)h(K+1)),
\end{align}
independent of the execution time $T$.
Notice that if $h(x) = 1/(x+1),$ then the above performance measure becomes the constant $(K(K+2)-f)/((K+1)(K+2)).$

\subsection{Loss probabilities}
We now examine the asymptotic behavior of the loss probability when the buffer capacity at each server is $B<\infty$ and the arrival rate $\lambda(N)$ satisfies~\eqref{eq:f=0} with $K=B$.
We will establish lower and upper bounds, and prove that these asymptotically coincide.
When the buffer capacity $B$ is finite, to characterize the asymptotic steady-state loss probability of the JSQ$(d(N))$ scheme, we bound it from below and above by that of an ordinary and a modified Erlang loss system, respectively.
The lower and upper bounds rely on a stochastic comparison.

Suppose $Y_1(t)$ and $Y_2(t)$ are two non-explosive, continuous-time Markov processes taking values in a complete separable metric space $E$. 
Let $X_1(t)$ and $X_2(t)$ be two birth-death processes defined on the same probability space, with finite state spaces $\big\{0,1,\ldots,n_1\big\}$ and $\big\{0,1,\ldots,n_2\big\}$, whose birth rates are $f_1(X_1(t), Y_1(t))$ and $f_2(X_2(t), Y_2(t))$, and death rates are $g_1(X_1(t), Y_1(t))$ and $g_2(X_2(t), Y_2(t))$, respectively.
\begin{lemma}\label{lem:bdprocess-stoch-comp}
If $n_1\leq n_2$, and for all $x\in \big\{0,1,\ldots,n_1\big\}$, $f_1(x,y_1)\leq f_2(x,y_2)$ and $g_1(x,y_1)\geq g_2(x,y_2)$, for all $y_1, y_2\in E$, then 
$\big\{X_1(t)\big\}_{t\geq 0}\leq_{st} \big\{X_2(t)\big\}_{t\geq 0}$, provided $X_1(0)\leq_{st} X_2(0)$.
\end{lemma}
\begin{proof}
The proof is fairly straightforward, but we present it briefly for the sake of completeness.
First we suitably couple the two processes, and then
as before, using the forward induction on event times, we show that the inequality holds throughout the sample path.
Define the processes $\big(X_1(\cdot),X_2(\cdot),Y_1(\cdot),Y_2(\cdot)\big)$ on the same probability space.
Due to the assumptions in the theorem, we do not need any condition on the evolution of $Y_1$ and $Y_2$, provided that they are defined on the same probability space.
Maintain two exponential clocks of rate $M_B:=\max\big\{f_1(x_1,y_1),f_2(x_2,y_2)\big\}$ (birth-clock) and $M_D:=\max\big\{g_1(x_1,y_1),g_2(x_2,y_2)\big\}$ (death-clock), respectively.
When the birth-clock rings, draw a single uniform$[0,1]$ random variable $u$ say, and a birth occurs in the $X_1$ process and $X_2$ process if $u\leq f_1(x_1,y_1)/M_B$ and $u\leq f_2(x_2,y_2)/M_B$, respectively.
Couple the deaths also, in a similar fashion.
Note that the processes thus constructed satisfy the relevant statistical laws in terms of the transition rates $f_1(x_1, y_1)$ and $f_2(x_2, y_2)$.

Now under the above coupling we prove the inequality. Assume that the inequality holds at event time $t_0$, and $X_1(t_0) = x_1$ and $X_2(t_0) = x_2$.
Note that if $x_1<x_2$, then trivially the inequality holds at the next event time $t_1$. 
Therefore, without loss of generality, assume $x_1= x_2=x\leq n_1$.
We will distinguish between two cases depending on whether the birth-clock or death-clock rings at time epoch $t_1$.
In the former case, observe that since $f_1(x,y_1)\leq f_2(x,y_2)$ for all $y_1, y_2\in E$, whenever there is a birth in the $X_1$ process, there will be a birth in the $X_2$ process as well.
Thus the inequality is preserved.
Alternatively, if the death-clock rings at time epoch $t_1$, then observe that since $g_1(x,y_1)\geq g_2(x,y_2)$ for all $y_1, y_2\in E$, whenever there is a death in the $X_2$ process, there will be a death in the $X_1$ process as well, and the inequality is preserved. This completes the proof.
\end{proof}

Denote by $\Er(C,\lambda)$ an Erlang loss system with
capacity $C$, load $\lambda$, and exponential service times with unit mean.
We further introduce a modified Erlang loss system $\hEr(n,d)$ with capacity $B(N-n)$, 
and arrival rate $\lambda$, with unit-exponential service times, where a fraction
$$p(n,d):=\left(1-\frac{n+1}{N}\right)^{d},$$ 
of tasks is rejected upfront,
independent of any other processes.
Note that the number of active tasks in the $\hEr(n,d)$ system evolves like an $\Er(B(N-n),\lambda p(n,d))$ system.

Define $C(N):=BN$, $\hat{C}(N):=B (N - n(N))$, and $\hat{\lambda}(N):= \lambda(N)p(n(N),d(N))$. 
Denote the total number of active tasks at time $t$ in the $N^{\mathrm{th}}$ system following the JSQ$(d(N))$ scheme, an $\Er(C(N),\lambda(N))$ system, and an $\hEr(n(N),d(N))$ system by $Y^{\sss d(N)}(t)$, $Y^N_\Er(t)$, and $Y^N_{\hEr}(t)$, respectively.
Denote the associated steady-state loss probabilities by $L^{\sss d(N)}$, $L(C,\lambda)$ and $\hat{L}(n,d)$, respectively.

\begin{lemma}\label{lem:stationary-blocking}
For all $N\geq 1$, $d(N)\geq 1$, and $n(N)<N$,
\begin{align*}
&\mathrm{(a)}\quad \big\{Y^N_{\hEr}(t)\big\}_{t\geq 0}\leq_{st} \big\{Y^{\sss d(N)}(t)\big\}_{t\geq 0} \leq_{st} \big\{Y^N_{\Er}(t)\big\}_{t\geq 0},\\
&\mathrm{(b)}\quad L(C(N),\lambda(N))\leq L^{\sss d(N)}\leq \hat{L}(n(N),d(N)).
\end{align*}
\end{lemma}
\begin{proof}
(a)
For the lower bound, observe that the rate of increase of the process $Y^{\sss d(N)}(\cdot)$ is at most that of the process $Y^N_\Er(\cdot)$, and the rate of decrease at any state is the same in both processes.
Thus, Lemma~\ref{lem:bdprocess-stoch-comp} implies that if both systems start from the same occupancy states, then $\big\{Y^{\sss d(N)}(t)\big\}_{t\geq 0}\leq_{st} \big\{Y^{N}_{\Er}(t)\big\}_{t\geq 0}$.
Consequently, in the steady state, $Y^{\sss d(N)}(\infty)\leq_{st} Y^N_{\Er}(\infty)$, and invoking Little's law yields $L(C(N),\lambda(N))\leq L^{\sss d(N)}.$

For the upper bound, 
first observe that at any arrival, as long as one of the $n(N)$ lowest-ordered server pools is sampled, which occurs with probability $1 - p(n(N), d(N))$, a task can only get lost when the total number of active tasks is at least $B (N - n(N))$.
Thus when the total number of active tasks $Y^{\sss d(N)}(\cdot)$ in the system under the JSQ$(d(N))$ scheme is $y$, the rate of increase of $Y^{\sss d(N)}(t)$ is at least $\lambda(N)(1 - p(n(N), d(N)))$ if $y \leq B (N - n(N))$,  and the rate of decrease is given by $y$.
Comparing with the modified Erlang loss system $\hEr(n(N),d(N))$ and using Lemma~\ref{lem:bdprocess-stoch-comp}, we obtain that if $Y^{\sss d(N)}(0)\geq_{st} Y^N_{\hEr}(0)$, then
$$\big\{Y^{\sss d(N)}(t)\big\}_{t\geq 0}\geq_{st} \big\{Y^N_{\hEr}(t)\big\}_{t\geq 0}.$$
The proof of the upper bound $L^{\sss d(N)}\leq \hat{L}(n(N),d(N))$ is then completed by again invoking Little's law.

(b) Little's law implies 
$$L^{\sss d(N)} = 1 - \frac{1}{\lambda(N)}\lim_{T\to\infty}\int_0^T Y^{\sss d(N)}(t)\dif t,$$
and similarly for the $\Er(C(N),\lambda(N))$ and $\hEr(n(N),d(N))$ systems. 
 Statement (b) then follows from statement (a).
\end{proof}

The proposition below states that the limiting loss probability for the JSQ$(d(N))$ scheme vanishes as long as $d(N)\to\infty$. 
\begin{proposition}
 For any $\lambda\leq B$, if $d(N)\to\infty$ as $N \to\infty$, then $L^{\sss d(N)}\to 0$, as $N\to\infty.$
\end{proposition}
\begin{proof}
From~\eqref{eq:nN-order} and~\eqref{eq:equiv-nN-dN}, we know if $d(N)\to\infty$, then 
there exists $n(N)$ such that as $N\to\infty$, $n(N)/N\to 0$ and $p(n(N),d(N))\to 0$.
For such a choice of $n(N)$, $\lambda(N)/C(N)\to \lambda/B\leq 1$, and $\hat{\lambda}(N)/\hat{C}(N)\to \lambda/B\leq 1$ as $N\to\infty$.
Therefore, using Lemma~\ref{lem:stationary-blocking} and the standard results of the Erlang loss function~\cite{J74}, we complete the proof of the proposition.
\end{proof}
\begin{remark}\normalfont
Note that in view of the results in~\cite{MKMG15, MMG15} for the JSQ$(d)$ schemes with fixed $d$, following the arguments as in Remark~\ref{rem:necessity-fluid}, the growth condition $d(N)\to\infty$ as $N\to\infty$ is also necessary to achieve an asymptotically zero probability of loss.
\end{remark}

We now further show that  the steady-state loss probability multiplied by $\sqrt{N}$ converges to a non-degenerate 
limit, which is the same as in an $\Er(C(N),\lambda(N))$ system.
The next theorem also establishes that if~\eqref{eq:f=0} is satisfied, and $d(N)/ (\sqrt{N}\log( N))\to 0$  as $N\to\infty$, then the steady-state loss probability is of higher order than $1/\sqrt{N}$. This indicates that  
the growth rate $\sqrt{N}\log (N)$ is not only sufficient but also nearly necessary. 

\begin{theorem}[{Scaled loss probability}]
\label{th:blocking}
Assume that $d(N)/ (\sqrt{N}\log (N))\to \infty$, as $N\to\infty$, and $\lambda(N)$ satisfies~\eqref{eq:f=0} with $K=B$. Then,  
\begin{equation}\label{eq:block-scale}
\lim_{N\to\infty}\sqrt{N}\ L^{\sss d(N)}=\frac{\phi(\beta)}{\sqrt{B}\Phi(\beta)},
\end{equation}
where $\phi(\cdot)$ and $\phi(\cdot)$ are the density and distribution function of the standard Normal distribution, respectively.
\end{theorem}
Since the right side of~\eqref{eq:block-scale} corresponds to the 
asymptotic steady-state loss probability in an $\Er(C(N),\lambda(N))$ system~\cite{J74, B76, W84}, we thus conclude that~\eqref{eq:block-scale} is optimal on $\sqrt{N}$-scale in terms of loss probability.
\begin{proof}[Proof of Theorem~\ref{th:blocking}]
 The idea again is to suitably bound the steady-state loss probability of the JSQ$(d(N))$ scheme.
Using Lemma~\ref{lem:stationary-blocking} and~\cite[Chapter 7, Theorem 15 (2)]{B76}, \cite{W84}, we obtain the lower bound as
\begin{equation}\label{eq:block-lower}
\begin{split}
&L^{\sss d(N)}\geq L(C(N),\lambda(N))\\
\implies & \varliminf_{N\to\infty}\sqrt{N}L^{\sss d(N)} \geq \varliminf_{N\to\infty} \sqrt{N} L(C(N),\lambda(N))= \frac{\phi(\beta)}{\sqrt{B}\Phi(\beta)}.
\end{split}
\end{equation}
For the upper bound, from~\eqref{eq:nN-order} and~\eqref{eq:equiv-nN-dN}, we know if $d(N)/(\sqrt{N}\log(N))\to\infty$ as $N\to\infty$, then there exists $n(N)$ with $n(N)/\sqrt{N}\to 0$ and 
\begin{equation}\label{eq:blockingupper2}
\sqrt{N}p(n(N),d(N))\to 0,\quad\mbox{as}\quad N\to\infty.
\end{equation}
Take such an $n(N)$. 
Again using~\cite[Chapter 7, Theorem 15 (2)]{B76}, we know that since
 as $N\to\infty$, $\hat{\lambda}(N)/\hat{C}(N)$ converges to one and $\hat{C}(N)/N$ converges to $B$,
\begin{equation}\label{eq:blockingupper3}
\lim_{N\to\infty}\sqrt{N}L(\hat{C}(N),\hat{\lambda}(N))= \frac{\phi(\beta)}{\sqrt{B}\Phi(\beta)}.
\end{equation}
Therefore, Lemma~\ref{lem:stationary-blocking}, and Equations~\eqref{eq:blockingupper2},~\eqref{eq:blockingupper3} yield
\begin{equation}
\varlimsup_{N\to\infty}\sqrt{N}L^{\sss d(N)} \leq \varlimsup_{N\to\infty} \sqrt{N} L(C(N),\lambda(N))+\varlimsup_{N\to\infty}\sqrt{N}p(n(N),d(N))= \frac{\phi(\beta)}{\sqrt{B}\Phi(\beta)}.
\end{equation}
Combination of the lower bound in~\eqref{eq:block-lower} and the above upper bound completes the proof.

\end{proof}

\begin{remark}[{Almost necessary condition for growth rate}]
\normalfont
It is worthwhile to mention that when $\lambda =K>0$ and $\lambda(N)$ satisfies~\eqref{eq:f=0}, the growth condition $d(N)/(\sqrt{N}\log(N))\to\infty$, as $N\to\infty$, is nearly necessary in order for the JSQ$(d(N))$ scheme to have the same diffusion limit as the ordinary JSQ policy.
More precisely, if  $d(N)/(\sqrt{N}\log(N))\to 0$ as $N\to\infty$, then the diffusion limit of the JSQ$(d(N))$ scheme differs from the ordinary JSQ policy.
In this remark we briefly sketch the outline of the proof. 
We will assume that the $d(N)$ server pools are chosen with replacement, to avoid cumbersome notation. But the proof technique and the result holds if the server pools are chosen without replacement.

\noindent
Assume on the contrary that as in the ordinary JSQ policy, if $N^{-1/2}(KN-\sum_{i=1}^KQ_i^{\sss d(N)}(0))$ is tight, then
 $N^{-1/2}(KN-\sum_{i=1}^KQ_i^{\sss d(N)}(t))$ is a stochastically bounded process.
 We argue that in this case, for any finite time $t$, the cumulative number of tasks joining a server with $K$ active tasks (or the cumulative number of lost tasks in case $K=B$) $L^{\sss d(N)}(t)$ does not scale with $\sqrt{N}$, and arrive at a contradiction. 
Indeed, $\big\{L^{\sss d(N)}(t)\big\}_{t\geq 0}$ admits the following martingale decomposition:
\begin{equation}\label{eq:Lmart-decomp}
L^{\sss d(N)}(t) = M_L^N(t) +\langle M_L^N\rangle(t),
\end{equation}
where $\big\{M_L^N(t)\big\}_{t\geq 0}$ is a martingale with compensator and predictable quadratic variation process given by 
$$\langle M_L^N\rangle(t)=\lambda(N) \int_0^t \left(Q_K^{\sss d(N)}(s-)/N\right)^{\sss d(N)}\dif s.$$
Since $\langle M_L^N\rangle(t)/N\leq \lambda t$, $\big\{M_L^N(t)/\sqrt{N}\big\}_{t\geq 0}$ is stochastically bounded.
We will show that $\langle M_L^N\rangle(t)$ is stochastically unbounded on $\sqrt{N}$-scale. 
From~\eqref{eq:Lmart-decomp}, this will imply that the process $\big\{L^{\sss d(N)}(t)/\sqrt{N}\big\}_{t\geq 0}$ is stochastically unbounded, which will complete the proof.
Note that 
\begin{align*}
Q_K^{\sss d(N)}(s)= N- (N-Q_K^{\sss d(N)}(s))\geq N- \sum_{i=1}^K(N-Q_i^{\sss d(N)}(s)),
\end{align*}
and hence,
\begin{align*}
 \langle M_L^N\rangle(t)&\geq \lambda(N) \int_0^t \left(1- \frac{1}{N}\sum_{i=1}^K(N-Q_i^{\sss d(N)}(s))\right)^{\sss d(N)}\dif s\\
&\geq \lambda(N) t \left(1- \frac{1}{N}\sup_{s\in [0,t]}\sum_{i=1}^K(N-Q_i^{\sss d(N)}(s))\right)^{\sss d(N)}.
\end{align*}
For any $T\geq 0$, since $\sup_{t\in[0,T]}\left(KN-\sum_{i=1}^KQ_i^{\sss d(N)}(t)\right)$ is $\Op(\sqrt{N})$, 
for any function $c(N)$ growing to infinity (to be chosen later), we have with probability tending to 1,
\begin{align*}
\frac{\lambda(N)T}{\sqrt{N}}\left(1-\frac{1}{N}\sup_{t\in[0,T]}\left(KN-\sum_{i=1}^KQ_i^{\sss d(N)}(t)\right)\right)^{\sss d(N)}&\geq \frac{\lambda(N)T}{\sqrt{N}}\left(1-\frac{\sqrt{N}c(N)}{N}\right)^{\sss d(N)}\\
&\geq  \frac{\lambda(N)T}{\sqrt{N}}\left(1-\frac{c(N)}{\sqrt{N}}\right)^{\sss d(N)}.
\end{align*}
Now since $d(N)/ \sqrt{N}\log (N)\to 0$ as $N\to\infty$, define $\omega(N):=\sqrt{N}\log (N)/d(N)$, which goes to infinity as $N$ grows large. Choose $c(N)$ such that $c(N)/\omega(N)\to 0$, as $N\to\infty.$ In that case,
\begin{align*}
\frac{\lambda(N)T}{\sqrt{N}}\left(1-\frac{c(N)}{\sqrt{N}}\right)^{\sss d(N)}
&= T\exp \left[\log (\sqrt{N}-\beta)+\frac{\sqrt{N}\log (N)}{\omega(N)}\log\left(1-\frac{c(N)}{\sqrt{N}}\right)\right]\\
& = T\exp \left[\log (\sqrt{N}-\beta)-\frac{\sqrt{N}\log (N)}{\omega(N)}\frac{c(N)}{\sqrt{N}}\right]\\
&\to\infty\quad\mbox{as }N\to\infty.
\end{align*}
\end{remark}

\section{Conclusion}\label{sec:conclusion}
In the present paper we have investigated asymptotic optimality
properties for JSQ$(d)$ load balancing schemes in large-scale systems.
Specifically, we considered a system of $N$ parallel identical
server pools and a single dispatcher which assigns arriving tasks
to the server with the minimum number of tasks among $d(N)$
randomly selected server pools.
We showed that the fluid limit in a regime where the total arrival
rate and number of server pools grow large in proportion coincides
with that for the ordinary JSQ policy ($d(N) = N$) as long as
$d(N) \to \infty$ as $N \to \infty$, however slowly.
We also proved that the diffusion limit in the Halfin-Whitt regime
corresponds to that for the ordinary JSQ policy as long as
$d(N)$ grows faster than $\sqrt{N} \log(N)$, and that the latter
growth rate is in fact nearly necessary.
These results indicate that the optimality of the JSQ policy can be
preserved at the fluid-level and diffusion-level while reducing
the communication overhead by nearly a factor $O(N)$
and $O(\sqrt{N} / \log(N))$, respectively.
In future work we plan to further establish convergence rates and extend
the results to non-exponential service requirement distributions.

The proofs of the asymptotic optimality properties rely on a novel
stochastic coupling construction to bound the difference in the
system occupancy processes between the JSQ policy and a JSQ($d$)
scheme with an arbitrary value of~$d$.
It is worth observing that the coupling construction is two-dimensional in nature, and fundamentally different from the classical coupling approach used for deriving stochastic dominance properties for the ordinary JSQ policy and for establishing universality in the single-server case~\cite{MBLW16-3}.
As it turns out, a direct comparison between the JSQ policy
and a JSQ($d$) scheme is a significant challenge.
Hence, we adopted a two-stage approach based on a novel class
of schemes which always assign the incoming task to one of the
server pools with the $n(N) + 1$ smallest number of tasks.
Just like the JSQ($d(N)$) scheme, these schemes may be thought
of as `sloppy' versions of the JSQ policy.
Indeed, the JSQ($d(N)$) scheme is guaranteed to identify the
server pool with the minimum number of tasks, but only among
a randomly sampled subset of $d(N)$ server pools.
In contrast, the schemes in the above class only guarantee that
one of the $n(N) + 1$ server pools with the smallest number of tasks
is selected, but across the entire system of $N$ server pools.
We showed that the system occupancy processes for an intermediate
blend of these schemes are simultaneously close on a $g(N)$ scale
(e.g.~$g(N) = N$ or $g(N) = \sqrt{N}$) to both the JSQ policy
and the JSQ($d(N)$) scheme for suitably chosen values of $d(N)$
and $n(N)$ as function of $g(N)$.
Based on the latter asymptotic universality, it then sufficed to
establish the fluid and diffusion limits for the ordinary JSQ policy.

\section*{Acknowledgment}
This research was financially supported by an ERC Starting Grant and by The Netherlands Organization for Scientific Research (NWO) through TOP-GO grant 613.001.012 and Gravitation Networks grant 024.002.003. Dr.~Whiting was supported in part by an Australian Research grant DP-1592400 and in part by a Macquarie University Vice-Chancellor Innovation Fellowship.

\bibliographystyle{apa}

\bibliography{bibl}


\end{document}